\documentclass[11pt]{article}%
\usepackage{amsmath}
\usepackage{amsfonts}
\usepackage{amssymb}
\usepackage{graphicx}
\usepackage{a4wide}
\usepackage{xcolor}%
\providecommand{\U}[1]{\protect\rule{.1in}{.1in}}
\newtheorem{theorem}{Theorem}[section]

\newtheorem{corollary}[theorem]{Corollary}

\newtheorem{example}[theorem]{Example}

\newtheorem{lemma}[theorem]{Lemma}

\newtheorem{proposition}[theorem]{Proposition}
\newtheorem{remark}{Remark}[section]

\newenvironment{proof}[1][Proof]{\noindent\textbf{#1.} }{\ \rule{0.5em}{0.5em}}
\definecolor{color0}{gray}{.50}
\definecolor{color1}{rgb}{0,.2,.8}
\definecolor{color2}{rgb}{1,.2,0}
\definecolor{color3}{rgb}{.8,.5,1}

\begin{document}

\title{Minimal element theorems revisited}
\author{Andreas H. Hamel\thanks{Free University of Bozen-Bolzano, Faculty of Economics and Management, University Square 1, 39031 Bruneck-Brunico, Italy}, Constantin Z\u{a}linescu\thanks{University "Al. I. Cuza" Iasi, Faculty of Mathematics, Bd. Carol I, 11, 700506 Iasi, Romania}}
\maketitle

\begin{abstract}
Starting with the Brezis-Browder principle, we give stronger versions of many
variational principles and minimal element theorems which appeared in the
recent literature. Relationships among the elements of different sets of
assumptions are discussed and clarified, i.e., assumptions to the metric
structure of the underlying space and boundedness assumptions. New results
involving set-valued maps and the increasingly popular set relations are
obtained along the way.
\end{abstract}




{\color{color1} }


\section{Introduction}

The celebrated variational principle due to Ekeland ensures the existence of
minimal elements with respect to (wrt for short) an order relation on a
complete metric space $(X, d)$ generated by a lower semicontinuous function
$f:X \to\mathbb{R}\cup\{+\infty\}$ which is bounded from below; this order is
defined by
\begin{equation}
\label{EqBasicEVPOrder}x_{1} \preceq x_{2} \quad:\Longleftrightarrow\quad
f(x_{1}) + d(x_{1}, x_{2}) \leq f(x_{2}).
\end{equation}
The assumptions required from $f$ lead to certain features of $\preceq$
whereas the assumptions to $X$ admit a countable induction argument which
produces a Cauchy sequence which is decreasing wrt $\preceq$ and whose limit
then is the desired minimal element.

Quite some effort went into attempts to generalize this result.

In the first line of research, the assumptions to $X$, a (complete) metric
space, are basically kept, but the order relation \eqref{EqBasicEVPOrder} is
replaced by a more general one. A blueprint result of this type is
\cite[Theorem 3.2]{DanHegMed83}, others include \cite[Theorem 1]{Tur:81},
\cite[Theorem 16]{Ham:05} (for this, see also \cite[Theorem 2.2]{HamTam08})
and \cite[Lemma 2.2]{LiuNg11}. Such results can also be applied to order
relations of the type \eqref{EqBasicEVPOrder} with $f$ a vector- or even
set-valued function which makes them very powerful tools. The thesis
\cite{Ham:05} exemplifies this approach with corollaries for functions $f$
mapping in preordered monoids (of sets), for example. Below, Corollaries
\ref{c-hamel}, \ref{zv-tnew}, \ref{zv-c-liu-ng}, \ref{zv-t-turinici} are of
this type. One could also look at these results as attempts to separate the
countable induction argument already used in the proof of \cite[Theorem
1]{Eke:79} from the features of the order relation; applications consist in a
mere check if the order relation in question satisfies the required
assumptions which links it to the metric structure. Therefore, the mentioned
results can be proven without involving the Brezis--Browder principle as shown
in the alternative proof of Theorem \ref{t-hz} below.

The second line of research is concerned with order relations on arbitrary
sets (without a metric structure, for example). Of course, alternative
requirements have to be added. A prominent result of this type is the
Brezis--Browder principle \cite{BreBro:76} (BB--principle) where the existence
of a real-valued function is assumed which is bounded from below and
increasing wrt the order relation. Further examples for such results can be
found in \cite{Qiu:14, Qiu:16b} and also in \cite{Tur:14}. Of course, BB--type
theorems can also be used to obtain the results on metric spaces discussed in
the previous paragraph with a suitable monotone function; the proof of the
BB--principle also involves a countable induction argument (see the first
proof of Theorem \ref{ex-BB} below).

A third goal is to lift the results from order relations on a set $X$ to such
relations on a product set $X \times Z$. Results in this direction have been
obtained first by G\"{o}pfert and Tammer (see \cite[Section 3.10]%
{GopRiaTamZal:03} and the references 140--144 therein) and are usually called
minimal element theorems.

In this note, all three of the above lines are followed providing very general
results and discussing the relationships between different sets of assumptions
in detail. Consequently, most relevant results in the literature are obtained
as special cases. In particular, it is shown that variational principles and
minimal element theorems involving (very general) set relations can be
obtained (see, for example, Corollary \ref{p-z2} below). Such results can be
found in \cite{Ha05} for complete metric spaces (manuscript version
\cite{Ha:02} from 2002) and more general ones in \cite{HamLoh:06} (preprint
version \cite{HamLoh:02} from 2002), the latter reference already giving a
proof based on the BB-principle and nonlinear scalarization.


\section{Preliminary Notions and Results\label{sec-vz-10.1}}

Let $\mathbb{N}$ denote the set of natural numbers including $0$ and
$\mathbb{R}$ the real numbers. The set $\mathbb{R}_{+}$ is the set of all
nonnegative real numbers.

In the sequel $(X,d)$ is a metric space, $Y$ is a real separated topological
vector space, $Y^{\ast}$ is its topological dual, and $K\subseteq Y$ is a
proper convex cone; as usual, $K^{+}$ is the positive dual cone of $K$ and
$K^{\#}$ is the quasi-interior of $K^{+}$:%
\begin{align*}
K^{+}  &  =\{y^{\ast}\in Y^{\ast}\mid y^{\ast}(y)\geq0\ \forall\,y\in K\},\\
K^{\#}  &  =\{y^{\ast}\in Y^{\ast}\mid y^{\ast}(y)>0\ \forall\,y\in
K\backslash\{0\}\},
\end{align*}
where for $\alpha,\alpha^{\prime}\in\mathbb{R}$, $\alpha\leq\alpha^{\prime}$
(or, equivalently, $\alpha^{\prime}\geq\alpha$) means, as usual, that
$\alpha^{\prime}-\alpha\in\mathbb{R}_{+}$, while $\alpha<\alpha^{\prime}$ (or
$\alpha^{\prime}>\alpha$) means that $\alpha^{\prime}-\alpha\in\mathbb{R}%
_{+}\backslash\{0\}$.

If $Y$ is just a real linear space we (can) endow it with the finest locally
convex topology, that is, the core convex topology (see \cite[Exercise
2.10]{Hol:75}).

The preorder $\leq_{K}$ on $Y$ generated by $K$ via
\[
y_{1}\leq_{K}y_{2}\quad:\Longleftrightarrow\quad y_{2}\in y_{1}+K
\]
for $y_{1},y_{2}\in Y$ can be extended to the power set $2^{Y}$ by
\[
A_{1}\leq_{K}^{l}A_{2}\quad:\Longleftrightarrow\quad A_{2}\subseteq A_{1}+K,
\]
and it is easily seen that $\leq_{K}^{l}$ is reflexive, transitive, thus a
preorder, but not a partial order in general even if $\leq_{K}$ is
antisymmetric. See \cite{HamEtAl:15} for more details and references. Here and
in the following, the addition for sets is understood in the Minkowski
(element-wise) sense with the convention $\emptyset+A=A+\emptyset=\emptyset$
for all $A\in2^{Y}$. It can also be checked that $\leq_{K}^{l}$ is compatible
with this addition as well as with multiplication by non-negative numbers.

As in \cite{TamZal:11} and \cite{KhaTamZal:15b}, let $F:X\times
X\rightrightarrows K$ satisfy the conditions:

\begin{description}
\item[\textbf{(F1)}] \quad$0\in F(x,x)$ for all $x\in X$,

\item[\textbf{(F2)}] \quad$F(x_{1},x_{2})+F(x_{2},x_{3})\subseteq
F(x_{1},x_{3})+K$ for all $x_{1},x_{2},x_{3}\in X$.
\end{description}

Of course, (F2) is the triangle inequality for $F$ wrt $\leq_{K}^{l}$:
$F(x_{1},x_{3})\leq_{K}^{l}F(x_{1},x_{2})+F(x_{2},x_{3})$; moreover, when $K$
is pointed, (F1) is $F(x,x)\leq_{K}^{l}\{0\}$ since $F$ maps into $K$.

\begin{remark}
{\rm Obviously,
\begin{equation}
\label{EqLatticePreorder}A_{1} \leq^{l}_{K} A_{2} \quad\Longleftrightarrow
\quad A_{2} + K \subseteq A_{1} + K
\end{equation}
since $0 \in K$. The collection of sets $A \subseteq Y$ satisfying $A = A + K$
is denoted by $\mathcal{P}(Y, K)$, and $\leq^{l}_{K}$ coincides with
$\supseteq$ on $\mathcal{P}(Y, K)$. }

{\rm Therefore, in many cases (but not all), one can replace $F$ by
the
function $F_{K}\colon X\times X\rightrightarrows K$ defined by $F_{K}%
(x_{1},x_{2})=F(x_{1},x_{2})+K$. Indeed, if $F$ satisfies (F1) and (F2), then
$F_{K}$ does as well. In this case, $F_{K}(x,x)=K$ for all $x\in X$, and
$F_{K}$ is an order premetric in the sense of \cite[Def.~30]{Ham:05} mapping
into the preordered monoid $(\mathcal{P}(Y,K),+,\leq_{K}^{l})$ with the
Minkowski addition.}
\end{remark}

For $F$ satisfying conditions (F1) and (F2), and $z^{\ast}\in K^{+}$,
consider
\[
\eta_{F,z^{\ast}}:X\times X\rightarrow\overline{\mathbb{R}}_{+},\quad
\eta_{F,z^{\ast}}(x,x^{\prime}):=\inf\{z^{\ast}(z)\mid z\in F(x,x^{\prime
})\}.
\]
It follows immediately that
\[
\eta_{F,z^{\ast}}(x,x)=0\text{ ~and~~}\eta_{F,z^{\ast}}(x,x^{\prime\prime
})\leq\eta_{F,z^{\ast}}(x,x^{\prime})+\eta_{F,z^{\ast}}(x^{\prime}%
,x^{\prime\prime})\quad\forall x,x^{\prime},x^{\prime\prime}\in X.
\]
Using $F$ we introduce a preorder on $X\times2^{Y}$, denoted by $\preceq_{F}$,
in the following manner:
\begin{equation}
(x_{1},A_{1})\preceq_{F}(x_{2},A_{2}):\iff A_{1}+F(x_{1},x_{2}) \leq^{l}_{K}
A_{2} \iff A_{2}\subset A_{1}+F(x_{1} ,x_{2})+K; \label{zv-eq1-b}%
\end{equation}
clearly, $A_{2}=\emptyset\Rightarrow(x_{1},A_{1})\preceq_{F}(x_{2},A_{2})$,
and $[(x_{1},A_{1})\preceq_{F}(x_{2},A_{2})$, $A_{2}\neq\emptyset]$
$\Rightarrow$ $[A_{1}\neq\emptyset$, $F(x_{1},x_{2})\neq\emptyset]$. Indeed,
$\preceq_{F}$ is reflexive by (F1). Assume that $(x_{1},A_{1} )\preceq
_{F}(x_{2},A_{2})$ and $(x_{2},A_{2})\preceq_{F}(x_{3},A_{3})$. This means
$A_{1}+F(x_{1},x_{2}) \leq^{l}_{K} A_{2}$ and $A_{2}+F(x_{2},x_{3}) \leq
^{l}_{K} A_{3}$. Adding these two inequalities, applying (F2) and the
transitivity of $\leq^{l}_{K}$ yields $A_{1}+F(x_{1},x_{3}) \leq^{l}_{K}
A_{3}$ which means that $\preceq_{F}$ is transitive.

Of course,
\begin{equation}
(x_{1},A_{1})\preceq_{F}(x_{2},A_{2})\Longrightarrow A_{2}\subset A_{1}%
+K\iff:A_{1}\leq^{l}_{K} A_{2}; \label{zv-eq1-a}%
\end{equation}
moreover, by (F1), we have that
\begin{equation}
(x,A_{1})\preceq_{F}(x,A_{2})\iff A_{2}\subset A_{1}+K\iff A_{1}\leq^{l}_{K}
A_{2}. \label{zv-eq1-c}%
\end{equation}

Restricting $\preceq_{F}$ to $X\times\left\{  A\mid A\subset Y,\text{
}\operatorname*{card}A=1\right\}  $ we get the preorder $\preceq_{F}^{1}$ on
$X\times Y$ defined by
\begin{equation}
(x_{1},y_{1})\preceq_{F}^{1}(x_{2},y_{2})\quad:\iff\quad y_{2}\in
y_{1}+F(x_{1},x_{2})+K. \label{zv-eq1-b0}%
\end{equation}
Hence, for $x$, $x_{1}$, $x_{2}\in X$ and $y_{1}$, $y_{2}\in Y$, we have that
\begin{align}
(x_{1},y_{1})  &  \preceq_{F}^{1}(x_{2},y_{2})\Longrightarrow y_{1}\leq
_{K}y_{2},\label{zv-eq1-a0}\\
(x,y_{1})  &  \preceq_{F}^{1}(x,y_{2})\Longleftrightarrow y_{1}\leq_{K}y_{2}.
\label{zv-eq1-c0}%
\end{align}

Notice that $\preceq_{F}^{1}$ is nothing else than $\preceq_{F}$ from
\cite{TamZal:11}.

Besides (F1) and (F2) we shall consider also the condition

\begin{description}
\item[(F3)] \quad there exists $z_{F}^{\ast}\in K^{+}$ such that
\begin{equation}
\eta(\delta):=\inf z_{F}^{\ast}(F_{\delta}):=\inf\left\{  z_{F}^{\ast}(v)\mid
v\in F_{\delta}\right\}  >0\quad\forall\delta>0, \label{zv-f2}%
\end{equation}

\end{description}

\noindent where
\begin{equation}
F_{\delta}:=\cup\{F(x,x^{\prime})\mid x,x^{\prime}\in X,\ d(x,x^{\prime}%
)\geq\delta\} \label{zv-fdelta}%
\end{equation}
for $\delta\geq0$; it follows that $0\leq\eta(\delta^{\prime})\leq\eta
(\delta)$ for $0\leq\delta^{\prime}<\delta$ because $F_{\delta}\subseteq
F_{\delta^{\prime}}$ in such a case.

Clearly, condition (F3) can be rewritten as
\begin{equation}
\exists z_{F}^{\ast}\in K^{+},\ \forall\delta>0\ :\ \inf z_{F}^{\ast
}(F_{\delta})>0. \label{zv-f2b}%
\end{equation}
A weaker condition is
\begin{equation}
\forall\delta>0,\ \exists z^{\ast}\in K^{+}\ :\ \inf z^{\ast}(F_{\delta})>0.
\label{zv-f4}%
\end{equation}
An even weaker condition is the following%
\begin{equation}
\forall\delta>0,\ \forall(z_{n})\subseteq F_{\delta},\ \exists z^{\ast}\in
K^{+}\ :\ \limsup z^{\ast}(z_{n})>0; \label{zv-f5}%
\end{equation}
when (\ref{zv-f5}) holds Qiu \cite[Def.\ 3.5]{Qiu:13} says that $F$ \emph{is
compatible with} $d$. Clearly, if (F3) holds then $0\notin{\overline
{\operatorname*{conv}} }F(x,x^{\prime})$ for $x\neq x^{\prime}$. Condition
(F3) holds obviously if for some $z^{\ast}\in K^{+}$one has%
\begin{equation}
\forall x,x^{\prime}\in X:\inf_{z\in F(x,x^{\prime})}z^{\ast}(z)\geq
d(x,x^{\prime}). \label{zv-f3}%
\end{equation}
The set below is related to the above conditions:
\begin{equation}
K_{F}^{+}:=\{z^{\ast}\in K^{+}\mid z^{\ast}(z)>0~\forall z\in\cup_{\delta
>0}F_{\delta}\}. \label{zv-f2c}%
\end{equation}

In the sequel we consider a nonempty $K$--convex subset $H$ of $K$, that is
$H+K$ is convex. It is worth observing that
\begin{equation}
0\le\alpha\le\beta\Rightarrow\beta H\subset\alpha H+K,\qquad0\le\gamma
,\delta\Rightarrow\gamma H+\delta H \subset(\gamma+\delta) H+K. \label{r-hk}%
\end{equation}

The next result provides an important example of set-valued functions $F$
satisfying conditions (F1) and (F2) (see \cite[Lem.\ 10.1.1]{KhaTamZal:15b}).

\begin{lemma}
\label{zv-lem1}Let $\emptyset\neq H\subseteq K$ be a $K$--convex set.
Consider
\begin{equation}
F_{H}:X\times X\rightrightarrows K,\quad F_{H}(x,x^{\prime}):=d(x,x^{\prime
})H. \label{zv-eqfh}%
\end{equation}
Then

\emph{(i)} $F_{H}$ verifies (F1) and (F2).

\emph{(ii)} $F_{H}$ verifies condition (F3) iff $F_{H}$ verifies condition
(\ref{zv-f4}) iff there exists $z_{H}^{\ast}\in K^{+}$ such that $\inf
z_{H}^{\ast}(H)>0$; if $Y$ is a separated locally convex space, then $F_{H}$
verifies condition (F3) iff $0\notin\operatorname*{cl} (H+K)$. Moreover,
$F_{H}$ verifies condition (\ref{zv-f5}) iff
\begin{equation}
\forall(h_{n})\subseteq H,\ \exists z^{\ast}\in K^{+}\ :\ \limsup z^{\ast
}(h_{n})>0. \label{zv-f5b}%
\end{equation}

\end{lemma}

Clearly, if $H$ is $K$-convex, then $H+K$ also is. The function $F_{H} + K =
F_{H+K}$ satisfies (i) and (ii) of Lemma \ref{zv-lem1} and maps into
$\mathcal{P}(Y, K)$.

Note that the case $0\in H+K$ (that is $H\cap(-K)\neq\emptyset$) is not
interesting in the present context because (the metric of) $X$ is not
involved; indeed, in such a case $(x,A)\preceq_{F_{H}}(x^{\prime},A^{\prime})$
if and only if $A\leq_{K}^{l}A^{\prime}$. For this reason, in the sequel,
\emph{$H$ is a nonempty $K$--convex subset of $K\backslash(-K)$}. Moreover, we
denote by $\preceq_{H}$ and $\preceq_{H}^{1}$ the preorders $\preceq_{F_{H}}$
and $\preceq_{F_{H}}^{1}$, respectively.

\medskip

If (\ref{zv-f3}) holds for $z^{\ast}\in K^{+}$, then%
\begin{equation}
(x_{1},A_{1})\preceq_{F}(x_{2},A_{2})\Longrightarrow\inf z^{\ast}%
(A_{1})+d(x_{1},x_{2})\leq\inf z^{\ast}(A_{2}). \label{zv-eq3}%
\end{equation}
Indeed, from $(x_{1},A_{1})\preceq_{F}(x_{2},A_{2})$, we have that
$A_{2}\subset A_{1}+F(x_{1},x_{2})+K$, and so $z^{\ast}(A_{2})\subset z^{\ast
}(A_{1})+z^{\ast}(F(x_{1},x_{2}))+z^{\ast}(K)$; it follows that
\begin{align}
\inf z^{\ast}(A_{2})  &  \geq\inf\{z^{\ast}(A_{1})+z^{\ast}(F(x_{1}%
,x_{2}))+z^{\ast}(K)\}\nonumber\\
&  =\inf z^{\ast}(A_{1})+\inf z^{\ast}(F(x_{1},x_{2}))+\inf z^{\ast
}(K)\nonumber\\
&  =\inf z^{\ast}(A_{1})+\inf z^{\ast}(F(x_{1},x_{2}))\label{zv-eq3b}\\
&  \geq\inf z^{\ast}(A_{1})+d(x_{1},x_{2}),\nonumber
\end{align}
where, as usual, $\inf\emptyset:=+\infty$. Hence (\ref{zv-eq3}) holds. Using
(\ref{zv-eq3}) for $A_{1}\subset Y$ with $\inf z^{\ast}(A_{1})\in\mathbb{R}$,
we obtain that
\begin{equation}
\left[  (x_{1},A_{1})\preceq_{F}(x_{2},A_{2}),~(x_{2},A_{2})\preceq_{F}%
(x_{1},A_{1})\right]  \Longrightarrow\left[  x_{1}=x_{2},\ \inf z^{\ast}%
(A_{1})=z^{\ast}(A_{2})\right]  . \label{zv-eq4}%
\end{equation}
In fact (\ref{zv-eq4}) holds if $F$ verifies conditions (F1)--(F3) and
$A_{1}\subset Y$ is such that $\inf z_{F}^{\ast}(A_{1})\in\mathbb{R}$. Indeed
for $\delta:=d(x_{1},x_{2})$, from (\ref{zv-eq3b}) we get%
\[
\inf z_{F}^{\ast}(A_{2})\geq\inf z_{F}^{\ast}(A_{1})+\inf z_{F}^{\ast
}(F_{\delta}),\quad\inf z_{F}^{\ast}(A_{1})\geq\inf z_{F}^{\ast}(A_{2})+\inf
z_{F}^{\ast}(F_{\delta}),
\]
and so $\inf z_{F}^{\ast}(A_{2})\in\mathbb{R}$ and $-\left\vert \inf
z_{F}^{\ast}(A_{2})-\inf z_{F}^{\ast}(A_{1})\right\vert \geq\inf z_{F}^{\ast
}(F_{\delta})\geq0$.

Eq.\ (\ref{zv-eq4}) shows that $\preceq_{F}^{1}$ is antisymmetric (and so
$\preceq_{F}^{1}$ is a partial order) when $F$ verifies conditions (F1)--(F3)
with $z_{F}^{\ast}\in K^{\#}$. In what concerns $\preceq_{H}^{1}$ (for $H$ as
above, that is $H\subset K\backslash(-K)$ is nonempty and $K$--convex), we
have the following:%
\begin{equation}
(x_{1},y_{1})\preceq_{H}^{1}(x_{2},y_{2})\preceq_{H}^{1}(x_{1},y_{1}%
)\Longrightarrow x_{1}=x_{2}; \label{zv-eq2}%
\end{equation}
moreover, if $K$ is pointed, then $\preceq_{H}^{1}$ is antisymmetric. Indeed,
take $(x_{1},y_{1})\preceq_{H}^{1}(x_{2},y_{2})\preceq_{H}^{1}(x_{1},y_{1})$,
and assume $\alpha:=d(x_{1},x_{2})>0$. Then $\pm(y_{1}-y_{2})\in\alpha
H+K=\alpha(H+K)$, and so $0\in\alpha(H+K)$ by the convexity of $H+K$. Hence
$0\in H+K$, a contradiction. Therefore, (\ref{zv-eq2}) holds. Assume,
moreover, that $K$ is pointed and $(x_{1},y_{1})\preceq_{H}^{1}(x_{2}%
,y_{2})\preceq_{H}^{1}(x_{1},y_{1})$. Then, by (\ref{zv-eq2}) we have that
$x_{1}=x_{2}$, and so $\pm(y_{1}-y_{2})\in K$, whence $y_{1}=y_{2}$ because
$K$ is pointed. Hence $\preceq_{H}^{1}$ is antisymmetric.

For $F$ satisfying conditions (F1) and (F2), and $z^{\ast}\in K^{+}$, we
introduce the partial order $\preceq_{F,z^{\ast}}$ on $X\times2^{Y}$ by
\begin{equation}
(x_{1},A_{1})\preceq_{F,z^{\ast}}(x_{2},A_{2}):\iff\left\{
\begin{array}
[c]{l}%
(x_{1},A_{1})=(x_{2},A_{2})\text{ or}\\
(x_{1},A_{1})\preceq_{F}(x_{2},A_{2})\text{ and }\inf z^{\ast}(A_{1})<\inf
z^{\ast}(A_{2}).
\end{array}
\right.  \label{vz-eq5}%
\end{equation}
It is easy to verify that $\preceq_{F,z^{\ast}}$ is reflexive, transitive, and antisymmetric.

The restriction of $\preceq_{F,z^{\ast}}$ to $X\times Y$ is denoted by
$\preceq_{F,z^{\ast}}^{1}$ and is given by
\begin{equation}
(x_{1},y_{1})\preceq_{F,z^{\ast}}^{1}(x_{2},y_{2}):\iff\left\{
\begin{array}
[c]{l}%
(x_{1},y_{1})=(x_{2},y_{2})\text{ or}\\
(x_{1},y_{1})\preceq_{F}^{1}(x_{2},y_{2})\text{ and }z^{\ast}(y_{1})<z^{\ast
}(y_{2}).
\end{array}
\right.  \label{vz-eq6}%
\end{equation}
The partial order $\preceq_{F,z^{\ast}}^{1}$ was introduced and used in
\cite[p.\ 913]{GopTamZal:00} (see also \cite[p.\ 202]{GopRiaTamZal:03}) for
$F(x,x^{\prime}):=\{d(x,x^{\prime})k^{0}\}$ with $k^{0}\in K\setminus
(-\operatorname*{cl}K)$, while $\preceq_{F,z^{\ast}}$ was introduced and used
in \cite{TamZal:11}. A generalization can be found in \cite[Section 7.2,
p.131]{Ham:05} where $z^{\ast}$ was replaced by a general function
$f:Y\rightarrow\mathbb{R}\cup\{+\infty\}$ which is monotone wrt the second
component of an order relation on $X\times Y$. Such order relations are
basically two-component extensions of relations like the one defined in
(\ref{r-lefi}) below. Of course, for $z^{\ast}=0$ the relation $\preceq
_{F,z^{\ast}}$ reduces to the equality relation on $X\times2^{Y}$. If
$z^{\ast}\in K^{\#}$ and $0\notin F(x,x^{\prime})$ for $x\neq x^{\prime}$ one
has that $\preceq_{F,z^{\ast}}^{1}$ and $\preceq_{F}^{1}$ coincide.

We denote by $\preceq_{H,z^{\ast}}$ and $\preceq_{H,z^{\ast}}^{1}$ the partial
orders $\preceq_{F_{H},z^{\ast}}$ and $\preceq_{F_{H},z^{\ast}}^{1}$, respectively.


\section{The Brezis--Browder principle}

A basic preliminary result in getting EVP type results revealed to be the
Brezis--Browder principle. Consider $W$ a nonempty set and $t\subset W\times
W$ a transitive relation; the fact that $(w,w^{\prime})$ belongs to $t$ will
be denoted by $w\preccurlyeq w^{\prime}$ or $w^{\prime}\succcurlyeq w$.
Moreover, we shall identify $t$ and $\preceq$. As usual, we say that the
sequence $(w_{n})_{n\geq1}\subset W$ is $\preccurlyeq$--decreasing if
$w_{n+1}\preccurlyeq w_{n}$ for $n\geq1$; $(w_{n})_{n\geq1}\subset W$ is
strictly $\preccurlyeq$--decreasing if it is $\preccurlyeq$--decreasing and
$w_{n+1}\neq w_{n}$ for $n\geq1$. Moreover, we say that the nonempty set
$A\subset W$ is bounded from below (or lower bounded, or minorized) (wrt
$\preccurlyeq$) if there exists $v\in W$ such that $v\preccurlyeq w$ for all
$w\in A$. We set $S_{\preccurlyeq}(w):=\{w^{\prime}\in W\mid w^{\prime
}\preccurlyeq w\}$ for each $w\in W;$ when there is no danger of confusion we
write simply $S(w)$ instead of $S_{\preccurlyeq}(w)$. Of course,
$S_{\preccurlyeq}:W\rightrightarrows W$ is a multifunction whose domain is
$\operatorname*{dom}S_{\preccurlyeq}:=\{w\in W\mid S_{\preccurlyeq}%
(w)\neq\emptyset\}$. Clearly $w^{\prime}\in S_{\preccurlyeq}(w)\Rightarrow
S_{\preccurlyeq}(w^{\prime})\subset S_{\preccurlyeq}(w)$, and
$\operatorname*{dom}S=W$ when $\preceq$ is a preoder.

The next assumption will be used often in the sequel.

\begin{description}
\item[(Ab)] Any $\preccurlyeq$--decreasing sequence $(w_{n})_{n\geq1}\subset
W$ is minorized.
\end{description}

Observe that condition (Ab) is equivalent with the (apparently weaker) condition

\begin{description}
\item[(Ab')] Any strictly $\preccurlyeq$--decreasing sequence $(w_{n}%
)_{n\geq1}\subset W$ is minorized.
\end{description}

Indeed, clearly, (Ab) $\Rightarrow$ (Ab'). Conversely, assume that (Ab') holds
and take $(w_{n})_{n\geq1}\subset W$ a $\preccurlyeq$--decreasing sequence.
Set $P:=\{n\geq1\mid w_{n+1}\neq w_{n}\}$. If $P$ is finite, then there exists
$\overline{n}\geq1$ such that $w_{n+1}=w_{n}$ for $n\geq\overline{n}$, and so
$w_{n}=w_{\overline{n}}$ for $n\geq\overline{n};$ in this case
$w:=w_{\overline{n}}\leq w_{n}$ for all $n\geq1$. Assume that $P$ is infinite.
Then $P=\{n_{k}\mid k\geq1\}$ where $n_{k+1}>n_{k}$ for $k\in\mathbb{N}%
\setminus\{0\}$. Setting $w_{k}^{\prime}:=w_{n_{k}}$, we have that
$w_{k+1}^{\prime}\in S(w_{k}^{\prime})\setminus\{w_{k}^{\prime})$ for $k\geq
1$. By (Ab') there exists $w\in W$ such that $w\preccurlyeq w_{k}^{\prime}$
for $k\geq1$. But for $n\geq1$ there exists $k\geq1$ such that $n_{k}\geq n$,
and so $w\preccurlyeq w_{k}^{\prime}=w_{n_{k}}\preccurlyeq w_{n}$.

Assume that $\preceq$ is a transitive relation on $W$ and $\phi:W\rightarrow
\overline{\mathbb{R}}$. We say that $\phi$ is $\preceq$--increasing if
$w_{1}\preccurlyeq w_{2}$ implies $\phi(w_{1})\leq\phi(w_{2});$ $\phi$ is
strictly $\preceq$--increasing if $w_{1}\preccurlyeq w_{2}$ and $w_{1}\neq
w_{2}$ imply $\phi(w_{1})\leq\phi(w_{2})$.

Partial orders of the type $\preccurlyeq_{\phi}$ introduced in (\ref{r-lefi})
below will be used several times in the sequel.

\begin{lemma}
\label{lem-fi}Let $\preceq$ be a transitive relation on $W$, and let
$\phi:W\rightarrow\overline{\mathbb{R}}$ be $\preceq$--increasing. Let us set%
\begin{equation}
w_{1}\preccurlyeq_{\phi}w_{2}:\iff\big[w_{1}=w_{2},\text{ or }[w_{1}%
\preccurlyeq w_{2}\text{ and }\phi(w_{1})<\phi(w_{2})]\big].\label{r-lefi}%
\end{equation}
Then the following assertions hold:

\emph{(i)} $\preccurlyeq_{\phi}$ is a partial order on $W$ and $\phi$ is
strictly $\preceq_{\phi}$--increasing; moreover, $\preccurlyeq_{\phi}{}%
={}\preceq\cup{}\Delta_{W}$ whenever $\phi$ is strictly $\preceq$--increasing.

\emph{(ii)} If $\overline{w}\in W$ is a minimal point of $W$ wrt
$\preceq_{\phi}$, then $W\ni w^{\prime}\preccurlyeq\overline{w}$ implies
$\phi(w^{\prime})=\phi(\overline{w})$.

\emph{(iii)} If $(W,\preceq)$ verifies condition \emph{(Ab)}, then
$(W,\preccurlyeq_{\phi})$ verifies condition \emph{(Ab)}, too.
\end{lemma}

\begin{proof}
(i) From its very definition $\preccurlyeq_{\phi}$ is reflexive; moreover,
using the transitivity of $\preccurlyeq$ and the fact that $\phi$ is
$\preccurlyeq$--increasing, one gets immediately that $\preccurlyeq_{\phi}$ is
transitive. Now, take $w_{1}$, $w_{2}\in W$ such that $w_{1}\preccurlyeq
_{\phi}w_{2}$ and $w_{2}\preccurlyeq_{\phi}w_{1}$. Assuming that $w_{1}\neq
w_{2}$, from (\ref{r-lefi}) we get the contradiction $\phi(w_{1})<\phi
(w_{2})<\phi(w_{1})$. Hence $\preccurlyeq_{\phi}$ is also antisymmetric, and
so $\preccurlyeq_{\phi}$ is a partial order. The fact that $\phi$ is strictly
$\preceq_{\phi}$--increasing follows from the very definition of
$\preceq_{\phi}$. The equality $\preccurlyeq_{\phi}=\preceq\cup\Delta_{W}$ is
obvious when $\phi$ is strictly $\preceq$--increasing.

(ii) It is clear that $\phi:(W,\preccurlyeq_{\phi})\rightarrow\overline
{\mathbb{R}}$ is strictly $\preccurlyeq_{\phi}$--increasing. Consider a
minimal point $\overline{w}\in W$ wrt $\preceq_{\phi}$, and take $w^{\prime
}\preccurlyeq\overline{w}$; then $\phi(w^{\prime})\preccurlyeq\phi
(\overline{w})$. If $\phi(w^{\prime})<\phi(\overline{w})$ then $w^{\prime
}\preccurlyeq_{\phi}\overline{w}$ and $w^{\prime}\neq\overline{w}$,
contradicting the minimality of $\overline{w}$ wrt $\preccurlyeq_{\phi}$.
Hence $\phi(w^{\prime})=\phi(\overline{w})$.

(iii) Take $(w_{n})_{n\geq1}\subset W$ a $\preccurlyeq_{\phi}$--decreasing
sequence. Then $(w_{n})_{n\geq1}$ is $\preccurlyeq$--decreasing, and so there
exists $\widetilde{w}\in E$ such that $\widetilde{w}\preccurlyeq w_{n}$ for
$n\geq1$; hence $\phi(\widetilde{w})\leq\phi(w_{n})$ for $n\geq1$. If
$\phi(\widetilde{w})=\phi(w_{n_{0}})$ for some $n_{0}\geq1$ then
$w_{n}=w_{n_{0}}$ for $n\geq n_{0}$, and so $w:=w_{n_{0}}\preccurlyeq_{\phi
}w_{n}$ for $n\geq1$. In the contrary case, $\phi(\widetilde{w})<\phi(w_{n})$,
and so $\widetilde{w}\preccurlyeq_{\phi}w_{n}$, for $n\geq1$.
\end{proof}

\medskip

The next result is a slight extension of the celebrated Brezis--Browder principle.

\begin{theorem}
\label{ex-BB}\emph{(BB-principle)} Let $\preceq$ be a transitive relation on
$W$ satisfying \emph{(Ab)}, and let $\phi:W\rightarrow\mathbb{R}$ be a lower
bounded $\preceq$--increasing function. Then

\emph{(i)} for every $w\in\operatorname*{dom}S_{\preceq}$ there exists
$\overline{w}\in S_{\preceq}(w)$ such that $\phi(w^{\prime})=\phi(\overline
{w})$ for all $w^{\prime}\in S_{\preceq}(\overline{w});$

\emph{(ii)} moreover, if $\phi$ is strictly $\preceq$--increasing, then
$\preccurlyeq$ is antisymmetric and for every $w\in\operatorname*{dom}%
S_{\preceq}$ there exists $\overline{w}\in S_{\preceq}(w)$ such that
$S_{\preceq}(\overline{w})\subset\{\overline{w}\}$.
\end{theorem}

\begin{proof}
Using Lemma \ref{lem-fi} (i) and (iii), we have that $\preceq_{\phi}$ is a
partial order verifying (Ab) and $\phi$ is strictly $\preceq_{\phi}%
$--increasing. Let $w\in\operatorname*{dom}S_{\preceq};$ using the second part
of \cite[Cor.~1]{BreBro:76} (for the reversed order), there exists a minimal
point $\overline{w}$ of $W$ wrt $\preceq_{\phi}$ such that $\overline
{w}\preceq_{\phi}w$. By Lemma \ref{lem-fi} (ii), $W\ni w^{\prime}%
\preccurlyeq\overline{w}$ implies $\phi(w^{\prime})=\phi(\overline{w});$ hence
(i) holds. Assume moreover, that $\phi$ is strictly $\preceq$--increasing and
take $w^{\prime}\in S_{\preceq}(\overline{w})\setminus\{\overline{w}\}$ (if
possible). Then $w^{\prime}\preceq_{\phi}\overline{w}$, and so we get the
contradiction $w^{\prime}=\overline{w}$ by the $\preceq_{\phi}$-minimality of
$\overline{w}$. Hence (ii) holds, too.
\end{proof}

\medskip

As observed by Corneliu Ursescu (see \cite[p.~120]{Car:03}) and Turinici (see
\cite{Tur:14}), the BB-principle is valid also for $\preceq$--increasing
functions $\phi:(W,\preceq)\rightarrow\overline{\mathbb{R}}$. The
corresponding version of Theorem \ref{ex-BB} is the following result. For
getting it (as in \cite{Car:03} and \cite{Tur:14}) it is sufficient to
consider the strictly increasing function $\varphi:\overline{\mathbb{R}%
}\rightarrow\mathbb{R}$ defined by $\varphi(t):=t/\left(  1+\left\vert
t\right\vert \right)  $ for $t\in\mathbb{R}$, $\varphi(\pm\infty):=\pm1$ and
to apply Theorem \ref{ex-BB} for $\varphi\circ\phi:W\rightarrow\mathbb{R}$.

\begin{theorem}
\label{ex-eBB}\emph{(extended BB-principle)} Let $\preceq$ be a transitive
relation on $W$ satisfying \emph{(Ab)}, and let $\phi:W\rightarrow
\overline{\mathbb{R}}$ be a $\preceq$--increasing function. Then the
conclusions of Theorem \ref{ex-BB} hold.
\end{theorem}

Having in view several recent extensions of the BB-principle, we provide below
a proof of Theorem \ref{ex-BB}~(i) without using Lemma \ref{lem-fi} or
\cite[Cor.~1]{BreBro:76}.

\medskip

\emph{Alternative (direct) proof of Theorem \ref{ex-BB}~(i) }[assertion (ii)
follows immediately from (i)].\emph{ }

For $w\in W$ set $\phi_{w}:=\inf\{\phi(u)\mid u\in S(w)\}$ with the convention
$\inf\emptyset:=+\infty$. Clearly, $-\infty<\bar{\gamma}:=\inf\phi(W)\leq
\phi_{w}\leq\phi(w)$ for $w\in\operatorname*{dom}S$, and $\phi_{w}=\infty$ for
$w\notin\operatorname*{dom}S$; moreover, for $w^{\prime}\preccurlyeq w$ we
have that $\phi_{w^{\prime}}\geq\phi_{w}$. Consider $(\varepsilon_{n}%
)_{n\geq1}\subset{}]0,\infty\lbrack$ with $\varepsilon_{n}\rightarrow0$. Fix
$w\in\operatorname*{dom}S$ and set $w_{0}:=w$. If $\phi_{w_{0}}=\phi(w_{0})$,
$\overline{w}:=w_{0}$ is the desired element. In the contrary case, there
exists $w_{1}\in S(w_{0})$ such that $\phi(w_{1})<\min\{\phi(w_{0}%
),\phi_{w_{0}}+\varepsilon_{1}\}$. If $w_{1}\notin\operatorname*{dom}S$ or
$\phi_{w_{1}}=\phi(w_{1})$, $\overline{w}:=w_{1}$ is the desired element.
Continuing in this way, either the process stops at the step $n\geq1$ because
$w_{n}\notin\operatorname*{dom}S$ or $\phi_{w_{n}}=\phi(w_{n})$, in which case
$\overline{w}:=w_{n}$ is the desired element, or, else we get the (strictly)
$\preccurlyeq$--decreasing sequence $(w_{n})_{n\geq1}$ such that
\[
\phi(w_{n+1})<\min\{\phi(w_{n}),\phi_{w_{n}}+\varepsilon_{n+1}\}\quad\forall
n\geq0.
\]
Therefore, $(w_{n})_{n\geq0}$ is $\preccurlyeq$--decreasing, $(\phi_{w_{n}%
})_{n\geq0}$ is increasing and $\left(  \phi(w_{n})\right)  _{n\geq0}$ is
strictly decreasing (in $\mathbb{R}$) with
\[
\phi_{w_{n}}\leq\phi_{w_{n+1}}\leq\phi(w_{n+1})<\phi_{w_{n}}+\varepsilon
_{n+1}\quad\forall n\geq0.
\]
Hence
\[
\phi(w_{n+1})<\phi(w_{n})~\text{ and }~0\leq\phi(w_{n+1})-\inf\{\phi(u)\mid
u\in S(w_{n})\}<\varepsilon_{n+1}\quad\forall n\geq0.
\]
It follows that there exists $\gamma\in\mathbb{R}$ such that $\lim
_{n\rightarrow\infty}\phi_{w_{n}}=\lim_{n\rightarrow\infty}\phi(w_{n})=\gamma
$, and so $\phi_{w_{n}}\leq\gamma\leq\phi(w_{n})$ for $n\geq0$. By (Ab') we
get $\overline{w}\in W$ such that $\overline{w}\preccurlyeq w_{n}$, and so
$S(\overline{w})\cup\{\overline{w}\}\subset S(w_{n})\subset S(w_{0})$ for
every $n\geq0$; in particular, $\overline{w}\in S(w_{0})$. Hence for
$w^{\prime}\in S(\overline{w})\cup\{\overline{w}\}$ we have that $\phi_{w_{n}%
}\leq\phi(w^{\prime})\leq\phi(w_{n})$ for all $n\geq1$. Taking the limit for
$n\rightarrow\infty$ we get $\phi(w^{\prime})=\phi(\overline{w})$ $(=\gamma)$
for $w^{\prime}\in S(\overline{w})$. Hence $\overline{w}$ is the desired element.

\medskip

The above proof of Theorem \ref{ex-BB} shows that the next result holds true.

\begin{theorem}
\label{r-Qiu0}Let $\preceq$ be a transitive relation on $W$ and let
$\phi:W\rightarrow\mathbb{R}$ be a lower bounded $\preceq$--increasing
function. Assume that

\begin{description}
\item[(AQ)] any strictly $\preccurlyeq$--decreasing sequence $(w_{n})_{n\geq
1}\subset W$ with $\left(  \phi(w_{n})\right)  _{n\geq1}$ strictly decreasing
and $\phi(w_{n+1})-\inf\{\phi(v)\mid v\preccurlyeq w_{n}\}\rightarrow0$ is minorized.
\end{description}

Then assertion \emph{(i)} of Theorem \ref{ex-BB} holds. Moreover, if $\phi$ is
strictly $\preceq$--increasing, then assertion \emph{(ii)} of Theorem
\ref{ex-BB} holds, too.
\end{theorem}

\begin{remark}
\label{r-Qiu1} {\rm (a) Theorem 5.1 in \cite{Car:03} follows from
the first part of Theorem \ref{r-Qiu0} applying it for the reversed
order and $\phi :=-S$, because in \cite{Car:03} it is assumed that
$\preccurlyeq$ is a partial order and any increasing sequence
$(x_{n})_{n\geq1}$ with $(S(x_{n}))_{n\geq 1}$ strictly increasing
is bounded above. }

{\rm (b) Theorems 2.1 in \cite{Qiu:14}, \cite{Qiu:14b} and
\cite{Qiu:16b} follow applying Theorem \ref{r-Qiu0} for
$W:=S(x_{1})$ with $x_{1}\in S(x_{0})\backslash\{x_{0}\}$ such that
$\eta(x_{1})<\infty$ and $\phi :=\eta|_{W}$ because it is assumed
that every strictly $\preccurlyeq
$--decreasing sequence $(w_{n})_{n\geq1}\subset W$ with $\phi(w_{n+1}%
)-\inf\{\phi(v)\mid v\preccurlyeq w_{n}\}\rightarrow0$ is minorized. More
precisely, \cite[Th.~2.1]{Qiu:14} and \cite[Th.~2.1]{Qiu:16b} follow directly
using the second part of Theorem \ref{r-Qiu0} because $\phi=\eta|_{W}$ is
strictly increasing. In the case of \cite[Th.~2.1]{Qiu:14b}, using the first
part of Theorem \ref{r-Qiu0}, one obtains some $\overline{w}\in W$ such that
$\phi(w^{\prime})=\phi(\overline{w})$ for $w^{\prime}\in S(\overline{w})$.
Because, by condition (B), for $z_{1},z_{2}\in S(x)$ with $x\in S(x_{0})$ we
have that $\min\{\eta(z_{1}),\eta(z_{2})\}<\eta(x)$, $S(\overline{w})$
contains at most one element $\widetilde{w}$. If $S(\overline{w}%
)\subset\{\overline{w}\}$ we take $\widehat{w}:=\overline{w}$. In the contrary
case $\widetilde{w}\neq\overline{w}$ and $S(\overline{w})=\{\widetilde{w}\}$,
whence $S(\widetilde{w})\subset S(\overline{w})=\{\widetilde{w}\}$. Taking
$\widehat{w}:=\widetilde{w}$, the conclusion of \cite[Th.~2.1]{Qiu:14b}
holds.}
\end{remark}

As O. C\^{a}rj\u{a} mentions in \cite[p.~120]{Car:03}, the use of strict
monotonicity of the sequence $\left(  S(x_{n})\right)  $ in the condition
corresponding to (Ab) of the usual version of the BB-principle is due to
Corneliu Ursescu.

\smallskip

In the sequel $\preceq$ is a preorder if not stated explicitly otherwise.

\smallskip

Having the metric space $(X,d)$, the nonempty set $Z$, and $\mathcal{A}$ a
nonempty subset of $X\times Z$ preordered by $\preceq$, consider the conditions:

\begin{description}
\item[(C0)] $\forall\left(  (x_{n},z_{n})\right)  _{n\geq1}\subset\mathcal{A}$
$\preceq$--decreasing : $(x_{n})_{n\geq1}$ is Cauchy and $\exists
(x,z)\in\mathcal{A}$ such that $(x,z)\preceq(x_{n},z_{n})$ $\forall n\geq1$.

\item[(C'0)] $\forall\left(  (x_{n},z_{n})\right)  _{n\geq1}\subset
\mathcal{A}$ $\preceq$--decreasing : $\exists(x,z)\in\mathcal{A}$ such that
$x_{n}\rightarrow x$ and $(x,z)\preceq(x_{n},z_{n})$ $\forall n\geq1$.

\item[(Ca)] $\forall\left(  (x_{n},z_{n})\right)  _{n\geq1}\subset\mathcal{A}$
$\preceq$--decreasing : $(x_{n})_{n\geq1}$ is Cauchy.

\item[(Cb)] $\forall\left(  (x_{n},z_{n})\right)  _{n\geq1}\subset\mathcal{A}$
$\preceq$--decreasing : $\exists(x,z)\in\mathcal{A}$ such that $(x,z)\preceq
(x_{n},z_{n})$ $\forall n\geq1$.

\item[(C1)] $\forall\left(  (x_{n},z_{n})\right)  _{n\geq1}\subset\mathcal{A}$
$\preceq$--decreasing with $(x_{n})_{n\geq1}$ Cauchy : $\exists(x,z)\in
\mathcal{A}$ such that $(x,z)\preceq(x_{n},z_{n})$ $\forall n\geq1$.

\item[(C'1)] $\forall\left(  (x_{n},z_{n})\right)  _{n\geq1}\subset
\mathcal{A}$ $\preceq$--decreasing with $x_{n}\rightarrow x\in X$ : $\exists
z\in Z$ such that $(x,z)\in\mathcal{A}$ and $(x,z)\preceq(x_{n},z_{n})$
$\forall n\geq1$.

\item[(Ca1)] $\forall\left(  (x_{n},z_{n})\right)  _{n\geq1}\subset
\mathcal{A}$ $\preceq$--decreasing with $(x_{n})_{n\geq1}$ Cauchy : $(x_{n})$
is convergent.
\end{description}

Clearly, condition (Cb) is nothing else than condition (Ab) written for
$(\mathcal{A},\preceq)$. Furthermore, (Ca1) is automatically verified whenever
$(X,d)$ is complete.

When $X$ is a singleton $\{x_{0}\}$ the conditions above, excepting (Ca) and
(Ca1) which are automatically satisfied, reduce to (Ab) for $(\mathcal{A}
,\preceq)$, or, equivalently, for $(W,\preccurlyeq)$, where $W=\Pr
_{Z}(\mathcal{A}):=\{z \in Z \mid\exists x \in X \colon(x,z) \in A\}$ and
$w\preccurlyeq w^{\prime}$ if $(x_{0},w)\preceq(x_{0},w^{\prime})$ for
$w,w^{\prime}\in W$.

When $Z$ is a singleton $\{z_{0}\}$, the preorder $\preceq$ on $\mathcal{A}%
\subset X\times Z$ is uniquely determined by the preorder $\preccurlyeq$ on
$\Pr_{X}(\mathcal{A})$ defined by $x\preccurlyeq x^{\prime}$ if $(x,z_{0}%
)\preceq(x^{\prime},z_{0})$. So, when $Z$ is a singleton, the conditions above
reduce to the following ones on the preordered metric space $(X,d,\preccurlyeq
)$ (replacing $X$ by $\Pr_{X}(\mathcal{A})$ if necessary):

\begin{description}
\item[(A0)] $\forall(x_{n})_{n\geq1}\subset X$ $\preccurlyeq$--decreasing :
$(x_{n})_{n\geq1}$ is Cauchy and $\exists x\in X$ such that $x\preccurlyeq
x_{n}$ $\forall n\geq1$.

\item[(A'0)] $\forall\left(  x_{n}\right)  _{n\geq1}\subset X$ $\preccurlyeq
$--decreasing : $\exists x\in X$ such that $x_{n}\rightarrow x$ and
$x\preccurlyeq x_{n}$ $\forall n\geq1$.

\item[(Aa)] $\forall\left(  x_{n}\right)  _{n\geq1}\subset X$ $\preccurlyeq
$--decreasing : $(x_{n})_{n\geq1}$ is Cauchy.

\item[(Ab)] $\forall\left(  x_{n}\right)  _{n\geq1}\subset X$ $\preccurlyeq
$--decreasing : $\exists x\in X$ such that $x\preccurlyeq x_{n}$ $\forall
n\geq1$.

\item[(A1)] $\forall\left(  x_{n}\right)  _{n\geq1}\subset X$ $\preccurlyeq
$--decreasing with $(x_{n})_{n\geq1}$ Cauchy : $\exists x\in X$ such that
$x\preccurlyeq x_{n}$ $\forall n\geq1$.

\item[(A'1)] $\forall\left(  x_{n}\right)  _{n\geq1}\subset X$ $\preccurlyeq
$--decreasing with $x_{n}\rightarrow x\in X$ : $x\preccurlyeq x_{n}$ $\forall
n\geq1$.

\item[(Aa1)] $\forall\left(  x_{n}\right)  _{n\geq1}\subset X$ $\preccurlyeq
$--decreasing with $(x_{n})_{n\geq1}$ Cauchy : $(x_{n})$ is convergent.
\end{description}

Taking into account Remark \ref{zv-rem7} below, condition (Aa) means that
$\preccurlyeq$ is regular in the sense of \cite[p.\ 68]{Ham:05}, while (A'1)
means that all lower sections of $X$ wrt $\preccurlyeq$ are $\preccurlyeq
$-lower closed. When $(X,d,\preccurlyeq)$ verifies (Aa1), in \cite[p.\ 67]%
{Ham:05} it is said that $(X,d)$ is $\preccurlyeq$--complete. Condition (Aa1)
is also related to dynamical completeness of $(X,d)$ in the sense of
\cite[Def.~3.1]{Qiu:14}.

\begin{proposition}
\label{rel-c}Among the conditions above we have the following relations:%
\begin{gather}
(\text{\emph{C0}})\Longleftrightarrow(\text{\emph{Ca}})\wedge(\text{\emph{Cb}%
})\Longleftrightarrow(\text{\emph{Ca}})\wedge(\text{\emph{C1}}),\label{r-c0}\\
(\text{\emph{Cb}})\Longrightarrow(\text{\emph{C1}}),\quad(\text{\emph{C'0}%
})\Longrightarrow(\text{\emph{C'1}}),\label{r-c01}\\
(\text{\emph{C'0}})\Longrightarrow(\text{\emph{C0}}),\quad(\text{\emph{C'1}%
})\wedge(\text{\emph{Ca1}})\Longrightarrow(\text{\emph{C1}}). \label{r-c0011}%
\end{gather}
In general, $($\emph{C'1}$)\not \Rightarrow ($\emph{C1}$)$, and the converse
implications in (\ref{r-c01}) and (\ref{r-c0011}) are not true, even if
$(X,d)$ is complete and $Z$ is a singleton.
\end{proposition}

\begin{proof}
The implications in (\ref{r-c0}), (\ref{r-c01}) and (\ref{r-c0011}) are almost
obvious. The fact that the implication $($C'1$)\Rightarrow($C1$)$ in general
is not true when $(X,d)$ is not complete is shown in Example \ref{ex-c}~(a).
The fact that the converse implications in (\ref{r-c01}) and (\ref{r-c0011})
are not true in general follows from Example \ref{ex-c}~(b) and (c), respectively.
\end{proof}

\begin{example}
\label{ex-c} \emph{(a)} Consider $X:={}]0,1]$ endowed with the metric $d$
defined by $d(x,x^{\prime}):=\left\vert x-x^{\prime}\right\vert $ and the
order $\preccurlyeq$ defined by $x\preccurlyeq x^{\prime}$ if $x^{\prime}%
-x\in\mathbb{R}_{+}$. Take $(x_{n})_{\geq1}\subset X$ a $\preccurlyeq
$--decreasing sequence with $x_{n}\rightarrow x\in X$. On one hand, since the
metric and order on $X$ are induced by the usual metric and order of
$\mathbb{R}$, we have that $x\leq x_{n}$ for $n\geq1$, and so (A'1) holds. On
the other hand, the sequence $(1/n)_{n\geq1}\subset X$ is Cauchy and
$\preccurlyeq$--decreasing, but assuming that there exists $x\in X$ such that
$x\preccurlyeq1/n$ (thus $x\leq1/n$) for $n\geq1$, we get the contradiction
$x\leq0$. Hence (A1) does not hold.

\emph{(b)} Consider $X:=[0,1]\subset\mathbb{R}$, $d:X\times X\rightarrow
\mathbb{R}$ defined by $d(x,x^{\prime}):=0$ for $x=x^{\prime}$, $d(x,x^{\prime
}):=1$ for $x\neq x^{\prime}$, and $x\preccurlyeq x^{\prime}$ if $x^{\prime
}-x\in\mathbb{R}_{+}$ for $x,x^{\prime}\in X;$ clearly, $(X,d)$ is a complete
metric space. It is obvious that (A1) and (A'1) hold because $(x_{n})_{n\geq
1}\subset X$ is Cauchy (convergent) iff $(x_{n})_{n\geq1}$ is constant for
large $n$, but neither (A0), nor (A'0), holds; just take the sequence
$(x_{n})_{n\geq1}:=(1/n)_{n\geq1}\subset X$.

\emph{(c)} Consider $X:=\{-1\}\cup\lbrack0,1]\subset\mathbb{R}$, $d:X\times
X\rightarrow\mathbb{R}$ defined by $d(x,x^{\prime}):=\left\vert x-x^{\prime
}\right\vert $, and $x\preccurlyeq x^{\prime}$ if either $x^{\prime}=0$, or
$x,x^{\prime}\in X\backslash\{0\}$ and $x^{\prime}-x\in\mathbb{R}_{+}$ for
$x,x^{\prime}\in X;$ clearly, $(X,d)$ is a complete metric space. Because
$-1\preccurlyeq x$ for every $x\in X$, (A1) clearly holds. Also (A0) holds.
Indeed, take $(x_{n})_{n\geq1}\subset X$ a $\preccurlyeq$--decreasing
sequence. As observed above, $-1\preccurlyeq x_{n}$ for every $n\geq1$. If
$x_{n_{0}}\neq0$ for some $n_{0}\geq1$, then $(x_{n})_{n\geq n_{0}}\subset
X\backslash\{0\}$, and so $(x_{n})_{n\geq n_{0}}$ as a sequence in
$\mathbb{R}$ with its usual order is decreasing and bounded from below by
$-1$, and so it is convergent to an element in $X$ (hence $(x_{n})_{n\geq1}$
is Cauchy in $(X,d)$). In the contrary case, $x_{n}=0$ for $n\geq1$, and so
$(x_{n})_{n\geq1}$ is Cauchy. Moreover, (A'1) and (A'0) do not hold because
$(1/n)_{n\geq1}\subset X$ is $\preccurlyeq$--decreasing and converges to $0$,
but $0\not \preccurlyeq 1/n$ for every $n\geq1$.
\end{example}

Using Theorem \ref{ex-eBB} we get the next result; it is close to
\cite[Th.\ 21]{Ham:05}, and is the prototype of the results in this paper,
some of them being direct consequences of it.

\begin{theorem}
\label{t-hz}Let $(X,d)$ be a metric space, $Z$ a nonempty set, and let
$\mathcal{A}\subset X\times Z$ be a nonempty set preordered by $\preceq$ which
verifies \emph{(C0)}. Then for every $(x,z)\in\mathcal{A}$ there exists
$(\overline{x},\overline{z})\in\mathcal{A}$ such that $(\overline{x}%
,\overline{z})\preceq(x,z)$ and $\mathcal{A}\ni(x^{\prime},z^{\prime}%
)\preceq(\overline{x},\overline{z})$ $\Rightarrow$ $x^{\prime}=\overline{x}$.
\end{theorem}

\begin{proof}
For $(x,z)\in\mathcal{A}$ set $S(x,z):=\{(x^{\prime},z^{\prime})\in
\mathcal{A}\mid(x^{\prime},z^{\prime})\preceq(x,z)\}$; moreover, consider
\[
\phi:\mathcal{A}\rightarrow\overline{\mathbb{R}},\quad\phi
(x,z):=\operatorname*{diam}P_{X}\left(  S(x,z)\right)  .
\]
Since $S(x^{\prime},z^{\prime})\subset S(x,z)$ when $(x,z)$, $(x^{\prime
},z^{\prime})\in\mathcal{A}$ with $(x^{\prime},z^{\prime})\preceq(x,z)$,
$\phi$ is $\preceq$--increasing. Because of condition (Ab), the hypotheses of
Theorem \ref{ex-eBB} are verified.

Take $(x,z)\in\mathcal{A}$. Using Theorem \ref{ex-eBB} we get $(\overline
{x},\overline{z})\in\mathcal{A}$ such that $(\overline{x},\overline{z}%
)\preceq(x,z)$ and $\mathcal{A}\ni(x^{\prime},z^{\prime})\preceq(\overline
{x},\overline{z})$ $\Rightarrow$ $\phi(x^{\prime},z^{\prime})=\phi
(\overline{x},\overline{z})=:\alpha\in\lbrack0,\infty]$. We claim that
$\alpha=0$, and so the conclusion holds.

Assume that $\alpha>0$ and take $0<\beta<\tfrac{1}{2}\alpha$. Then there
exists $(x_{1},z_{1})\in S(\overline{x},\overline{z})$ such that
$d(x_{1},\overline{x})>\beta$; else, $d(x^{\prime},\overline{x})\leq\beta$ for
$(x^{\prime},z^{\prime})\in S(\overline{x},\overline{z})$, and so
$\phi(\overline{x},\overline{z})=\operatorname*{diam}\Pr_{X}\left(
S(\overline{x},\overline{z})\right)  \leq2\beta<\alpha$, a contradiction.
Since $(x_{1},z_{1})\preceq(\overline{x},\overline{z})$, $\operatorname*{diam}
\Pr_{X}\left(  S(x_{1},z_{1})\right)  =\phi(x_{1},z_{1})=\alpha$. As before,
we find $(x_{2},z_{2})\preceq(x_{1},z_{1})$ $(\preceq(\overline{x}%
,\overline{z}))$ with $d(x_{2},x_{1})>\beta$. Continuing in this way we get
the $\preceq$--decreasing sequence $\left(  (x_{n},z_{n})\right)  _{n\geq
1}\subset\mathcal{A}$ such that $d(x_{n+1},x_{n})>\beta$ for all $n\geq1$. By
(C0), $(x_{n})_{n\geq1}$ is Cauchy, and so we get the contradiction
$\beta<d(x_{n+1},x_{n})\rightarrow0$. Hence $\alpha=0$, and so $(x^{\prime
},z^{\prime})\in S(\overline{x},\overline{z})$ implies $x^{\prime}%
=\overline{x}$.
\end{proof}

\medskip

\textit{Alternative proof of Theorem \ref{t-hz}} (without using BB-principle).

Fix $(x,z)\in\mathcal{A}$; with the notation above, we have to show that there
exists $(\bar x, \bar z)\in S(x,z)$ with $P_{X}\big(S(\bar x, \bar
z)\big)=\{\bar x\} $.

Starting with $(x_{1}, z_{1}):=(x,z) \in\mathcal{A}$, define a sequence
$\big((x_{n}, z_{n})\big)_{n\ge1} \subseteq\mathcal{A}$ by picking $x_{n+1}
\in P_{X}\big(S(x_{n}, z_{n})\big)$ satisfying
\begin{equation}
d(x_{n+1}, x_{n}) \geq\min\big(1,\sup\big\{d(x, x_{n})\mid x \in
P_{X}\big(S(x_{n}, z_{n})\big)\big\} - {1}/{n}\big), \label{ah2}%
\end{equation}
and $z_{n+1} \in Z$ such that $(x_{n+1}, z_{n+1}) \in S(x_{n},z_{n})$.

Since $\big((x_{n},z_{n})\big)_{n\geq1}$ is $\preceq$--decreasing, by (C0),
the sequence $(x_{n})_{n\geq1}$ is Cauchy and there exists $(\bar{x},\bar
{z})\in\mathcal{A}$ such that $(\bar{x},\bar{z})\preceq(x_{n},z_{n})$ for
$n\geq1$; in particular $(\bar{x},\bar{z})\preceq(x_{1},z_{1})=(x,z)$. It
follows that $d(x_{n},x_{n+1})\rightarrow0$, and so $d(x_{n},x_{n+1})<1$ for
large $n$. These together with (\ref{ah2}) imply that $\operatorname{diam}%
P_{X}\big(S(x_{n},z_{n})\big)\rightarrow0$. Since $S(\bar{x},\bar{z})\subset
S(x_{n},z_{n})$ for $n\geq1$, we obtain that $\operatorname{diam}%
P_{X}\big(S(\bar{x},\bar{z})\big)=0$. Hence $P_{X}\big(S(\bar{x},\bar
{z})\big)=\{\bar{x}\}$ because $(\bar{x},\bar{z})\in S(\bar{x},\bar{z})$. The
proof is complete.

\begin{lemma}
\label{lem2}Let $(X,d)$ be a metric space, $Z$ a nonempty set, and
$\mathcal{A}\subset X\times Z$ a nonempty set preordered by $\preceq$. Then
condition \emph{(Ca)} holds if and only if for every $\preceq$--decreasing
sequence $((x_{n},z_{n}))_{n\geq1}\subseteq\mathcal{A}$ the sequence $(x_{n})$
is asymptotic, that is $d(x_{n},x_{n+1})\rightarrow0$.
\end{lemma}

\begin{proof}
Of course, if $((x_{n},z_{n}))_{n\geq1}\subseteq\mathcal{A}$ verifies
condition (Ca) then $d(x_{n},x_{n+1})\rightarrow0$. Conversely, consider the
$\preceq$--decreasing sequence $((x_{n},z_{n}))_{n\geq1}\subseteq\mathcal{A}$
with $d(x_{n},x_{n+1})\rightarrow0$. Suppose that $(x_{n})$ is not Cauchy.
Then there exist $\delta>0$ and a strictly increasing sequence $(n_{l}%
)_{l\geq1}\subseteq\mathbb{N}^{\ast}$ such that $d(x_{n_{l+1}},x_{n_{l}}%
)\geq\delta$ for all $l\geq1$. Since $\preceq$ is transitive, $((x_{n_{l}%
},z_{n_{l}}))_{l\geq1}$ is $\preceq$--decreasing, and so $(x_{n_{l}})_{l\geq
1}$ is asymptotic, that is $d(x_{n_{l+1}},x_{n_{l}})\rightarrow0$. This
obvious contradiction ends the proof.
\end{proof}

\medskip

Note that one must add that $(X,d)$ is complete in \cite[Th.~21]{Ham:05}.

Indeed, take $X:=Y:=\mathbb{R}_{+}\backslash\{0\}$, $X$ endowed with the usual
metric and $X\times Y$ endowed with the order defined by $(x,y)\preceq
(x^{\prime},y^{\prime}):\iff\lbrack x\leq x^{\prime}$, $y\leq y^{\prime}]$;
take also $M:=\{(x,x)\mid x>0\}$. Clearly the hypothesis of \cite[Th.~21]%
{Ham:05} is satisfied, but its conclusion does not hold.

\medskip

Taking into account (\ref{r-c0011}), the next result is an obvious consequence
of Theorem \ref{t-hz}; it is stated in \cite[Th.\ 23]{Ham:05} under the
supplementary assumption that $(X,d)$ is complete.

\begin{corollary}
\label{c-hamel}Let $(X,d)$ be a metric space, $Z$ a nonempty set, and let
$\emptyset\neq\mathcal{A}\subset X\times Z$ be preordered by $\preceq$. Assume
that $(\mathcal{A}, \preceq)$ verifies condition \emph{(C'0)}. Then the
conclusion of Theorem \ref{t-hz} holds.
\end{corollary}

Another consequence of Theorem \ref{t-hz} is the following result.

\begin{corollary}
\label{zv-tnew}Let $(X,d)$ be a metric space, preordered by $\preccurlyeq$,
for which \emph{(A0)} holds. Then for every $x\in X$ there exists
$\overline{x}\in X$ such that $\overline{x}\preccurlyeq x$ and $X\ni
x^{\prime}\preccurlyeq\overline{x}$ $\Rightarrow$ $x^{\prime}=\overline{x}$.
\end{corollary}

\begin{proof}
Take $\mathcal{A}:=X\times\{0\}$ and set $(x,0)\preceq(x^{\prime},0)$ if
$x,x^{\prime}\in X$ and $x\preccurlyeq x^{\prime}$. Clearly $\preceq$ is a
preorder on $\mathcal{A}$ and condition (C0) in Theorem \ref{t-hz} is
verified. Applying Theorem \ref{t-hz} we get the conclusion.
\end{proof}

\medskip

From the preceding result we get immediately the next one formulated by Liu
and Ng in \cite[Lem.~2.2]{LiuNg11}.

\begin{corollary}
\label{zv-c-liu-ng}Let $(X,d)$ be a metric space, preordered by $\preccurlyeq
$, for which \emph{(A'0)} holds. Then for every $x\in X$ there exists
$\overline{x}\in X$ such that $\overline{x}\preccurlyeq x$ and $X\ni
x^{\prime}\preccurlyeq\overline{x}$ $\Rightarrow$ $x^{\prime}=\overline{x}$.
\end{corollary}

\begin{remark}
\label{r-znew}{\rm Corollary \ref{zv-tnew} is slightly more general
than Corollary \ref{zv-c-liu-ng}. To see this take
$X:=\{-1\}\cup\{1/n\mid n\in\mathbb{N}\backslash\{0\}\}$ endowed
with the metric $d$ defined by $d(x,x^{\prime}):=\left\vert
x-x^{\prime}\right\vert $ and the order
$\preccurlyeq$ defined by $x\preccurlyeq x^{\prime}$ if $x^{\prime}%
-x\in\mathbb{R}_{+}$. Clearly the sequence $(1/n)_{n\geq1}\subset X$ is
$\preccurlyeq$--decreasing, but not convergent.}
\end{remark}

As observed in \cite{LiuNg11}, Corollary \ref{zv-c-liu-ng} is a reformulation
of the next result stated by Hamel and Tammer in \cite[Th.~2.2]{HamTam08};
this result was stated previously, in equivalent forms, by Turinici in
\cite[Th.~1]{Tur:81} and by Hamel in \cite[Th.~16]{Ham:05}.

\begin{corollary}
\label{zv-t-turinici}Let $(X,d)$ be a metric space, preordered by
$\preccurlyeq$ such that $(X,d)$ is $\preccurlyeq$--complete, that is every
$\preccurlyeq$--decreasing Cauchy sequence $(x_{n})_{n\geq1}\subseteq X$ is
convergent. Assume that

\emph{(i)} $S(x):=\{x^{\prime}\in X\mid x^{\prime}\preccurlyeq x\}$ is
$\preccurlyeq$--lower closed for every $x\in X$, that is for every
$\preccurlyeq$--decreasing sequence $(x_{n})_{n\geq1}\subseteq S(x)$ with
$x_{n}\rightarrow u$ one has $u\in S(x)$, and

\emph{(ii)} any $\preccurlyeq$--decreasing sequence $(x_{n})_{n\geq1}\subseteq
X$ is asymptotic.

Then for every $x\in X$ there exists $\overline{x}\in X$ such that
$\overline{x}\preccurlyeq x$ and $S(\overline{x})=\{\overline{x}\}$.
\end{corollary}

\begin{remark}
\label{zv-rem7}{\rm Conditions (i) and (ii) in Corollary
\ref{zv-t-turinici} are equivalent to (A'1) and (Aa), respectively.}
\end{remark}

Indeed, assume that (i) in Corollary \ref{zv-t-turinici} holds and take
$(x_{n})_{n\geq1}\subseteq X$ a $\preccurlyeq$--decreasing sequence with
$x_{n}\rightarrow x\in X$. Since $(x_{n})_{n\geq p}\subseteq S(x_{p})$ and
$S(x_{p})$ is $\preccurlyeq$--lower closed, we have that $x\in S(x_{p})$, and
so $x\preccurlyeq x_{p}$ for every $p$. Conversely, assume that (A'1) holds
and take $(x_{n})_{n\geq1}\subseteq S(x)$ a $\preccurlyeq$--decreasing
sequence with $x_{n}\rightarrow u$. Then $u\preccurlyeq x_{1}\preccurlyeq x$,
whence $u\in S(x)$. The equivalence of (ii) and (Aa) follows immediately from
Lemma \ref{lem2}.

\medskip

The equivalence of (ii) and (Aa) from the previous remark is established in
\cite[Prop.\ 41]{Ham:05}.

Note that in Corollary \ref{zv-t-turinici} (and Corollary \ref{zv-c-liu-ng})
$\preccurlyeq$ is in fact anti-symmetric (as observed in \cite[Prop.\ 40]%
{Ham:05} and \cite[Prop.\ 2.1]{HamTam08}). Indeed, take $x$, $x^{\prime}\in X$
with $x\preccurlyeq x^{\prime}$ and $x^{\prime}\preccurlyeq x$. Then the
sequence $(x_{n})_{n\geq1}$ defined by $x_{2n}:=x$ and $x_{2n-1}:=x^{\prime}$
is $\preccurlyeq$--decreasing; by (ii) we get $d(x,x^{\prime})=d(x_{n}%
,x_{n+1})\rightarrow0$, and so $x=x^{\prime}$. \medskip

Note also that Corollary \ref{zv-t-turinici} is slightly more general than the
Dancs--Heged\"{u}s--Medvegyev Theorem (see \cite[Th. 3.1]{DanHegMed83}), in
which $(X,d)$ is assumed to be complete instead of being $\preccurlyeq
$--complete and $S(x)$ is assumed to be closed instead of being $\preccurlyeq$--closed.

\medskip

The next result is very easy to prove, so we omit its proof.

\begin{proposition}
\label{p-z1}Let $(Z,\preceq)$ be a preordered set. For any $A_{1}$, $A_{2}%
\in2^{Z}$ let us set%
\begin{equation}
A_{1}\preceq^{l}A_{2}~:\iff~\left[  \forall z_{2}\in A_{2},\ \exists z_{1}\in
A_{1}:z_{1}\preceq z_{2}\right]  . \label{r-leset}%
\end{equation}

Then $\preceq^{l}$ is a preorder on $2^{Z}$. Moreover, for $z_{1},z_{2}\in Z$,
and $A$, $A_{1}$, $A_{2}\in2^{Z}$ we have that
\[
\{z_{1}\}\preceq^{l}\{z_{2}\}\Leftrightarrow z_{1}\preceq z_{2},\quad\quad
Z\preceq^{l}A\preceq^{l}\emptyset,\quad\quad\big[
A_{2}\neq\emptyset,~A_{1}\preceq^{l}A_{2}\big] \Rightarrow A_{1}\neq
\emptyset.
\]

\end{proposition}

In this generality, the preorder $\preceq^{l}$ was probably
discussed for the first time by Hamel in \cite[(2.6)]{Ham:05}, at
least in the context of variational analysis/optimization. However,
there are many precursors: the interested reader is referred to
\cite{HamEtAl:15} for a more thorough survey and detailed
references.

An important example of the above construction, already discussed in
\cite{KurTanHa:97} on a topological vector space ordered by a convex cone with
non-empty interior, is the following: Let $Y$ be a real vector space and
$K\subset Y$ a convex cone. Recall that the preorder $\leq_{K}$ determined by
$K$ on $Y$ is defined by $y_{1}\leq_{K}y_{2}$ $:\iff$ $y_{2}-y_{1}\in K$. The
preorder on $2^{Y}$ corresponding to $\leq_{K}$ using the definition in
(\ref{r-leset}) is the one given by $A_{1}\leq_{K}^{l}A_{2}$ $:\iff$
$A_{2}\subset A_{1}+K$ (see Section \ref{sec-vz-10.1}). Of course, $\leq_{K}$
is a partial order on $Y$ if and only if $K$ is pointed (that is
$K\cap(-K)=\{0\}$), but $\leq_{K}^{l}$ is a partial order if and only if
$K=\{0\}$. Indeed, $\{0\}\leq_{K}^{l}K\leq_{K}^{l}\{0\}$.

Another application of Theorem \ref{ex-eBB} is the following result involving
the set relation $\preceq^{l}$.

\begin{corollary}
\label{p-z2}Let $(Z,\preceq)$ be a preordered set, and $\phi:(Z,\preceq
)\rightarrow\overline{\mathbb{R}}$ a $\preceq$--increasing function. Consider
the preorder $\preceq^{l}$ on $2^{Z}$ defined in (\ref{r-leset}). For $A_{1}$,
$A_{2}\in2^{Z}$ let us set%
\begin{equation}
A_{1}\preceq_{\phi}^{l}A_{2}~:\iff~\big[A_{1}=A_{2}\text{ or }[A_{1}%
\preceq^{l}A_{2}\text{ and }\inf\phi(A_{1})<\inf\phi(A_{2})]\big].
\label{r-lesetfi}%
\end{equation}
Then $\preceq_{\phi}^{l}$ is a partial order on $2^{Z}$. Furthermore, let
$\mathcal{A}\subset2^{Z}$ be a nonempty set such that $(\mathcal{A}%
,\preceq^{l})$ verifies condition \emph{(Ab).} Then for every $A\in
\mathcal{A}$ there exists a minimal set $\overline{A}\in\mathcal{A}$ wrt
$\preceq_{\phi}^{l}$ such that $\overline{A}\preceq_{\phi}^{l}A;$ moreover,
$\mathcal{A}\ni A^{\prime}\preceq^{l}\overline{A}$ implies $\inf\phi
(A^{\prime})=\inf\phi(\overline{A})$.
\end{corollary}

\begin{proof}
Consider the mapping
\[
\phi^{\prime}:2^{Z}\rightarrow\overline{\mathbb{R}},\quad\phi^{\prime
}(A):=\inf\phi(A),
\]
where, as usual, $\inf\emptyset:=+\infty$. We claim that $\phi^{\prime}$ is
$\preceq^{l}$--increasing. Indeed, take $A_{1}$, $A_{2}\in2^{Z}$ such that
$A_{1}\preceq^{l}A_{2}$. From the definition of $\preceq^{l}$, for $z_{2}\in
A_{2}$ there exists $z_{1}\in A$ with $z_{1}\preceq z_{2}$, and so
$\phi^{\prime}(A_{1})\leq\phi(z_{1})\leq\phi(z_{2})$. Since $z_{2}\in A_{2}$
is arbitrary, we get $\phi^{\prime}(A_{1})\leq\Phi^{\prime}(A_{2})$. Using
Lemma \ref{lem-fi} for $W:=2^{Z}$, $\preccurlyeq:=\preceq^{l}$ and $\phi$
replaced by $\phi^{\prime}$, we obtain that $\preceq_{\phi}^{l}$ is a partial
order on $2^{Z}$; moreover, $\phi^{\prime}:(2^{Z},\preceq_{\phi}%
^{l})\rightarrow\overline{\mathbb{R}}$ is clearly strictly $\preceq_{\phi}%
^{l}$--increasing. The rest of the conclusion follows from Lemma \ref{lem-fi}
taking $E:=\mathcal{A}$.
\end{proof}

\medskip

Theorem \ref{ex-BB} (as well as Theorem \ref{ex-eBB}) shows the usefulness of
finding a monotone function on a preordered set for proving the existence of
minimal (maximal) elements. This will be the main procedure for getting
variants of the Ekeland variational principle (EVP) for vector and set-valued functions.


\section{A few technical notions and results}

Throughout this section $X$, $Y$, $K$, $H$ are as in Section \ref{sec-vz-10.1}
(if not stated otherwise explicitly), that is, $(X,d)$ is a metric space, $Y$
is a separated topological vector space, $K\subset Y$ is a proper convex cone,
and $H\subset K\backslash(-K)$ is a nonempty $K$--convex set. As seen in
Section \ref{sec-vz-10.1}, $\preceq_{H}:=\preceq_{F_{H}}$ is a preorder on
$X\times2^{Y}$. We also consider a set $\mathcal{A}\subseteq X\times2^{Y}$
such that
\[
Y_{\mathcal{A}}:=\bigcup\big\{A\mid A\in{\Pr}_{2^{Y}}(\mathcal{A}%
)\big\}\neq\emptyset
\]
where $\Pr_{2^{Y}}(\mathcal{A}) = \{B \subseteq Y \mid\exists x \in X
\colon(x, B) \in\mathcal{A}\}$.

Eventually, we are interested in obtaining results similar to the one in
Theorem \ref{t-hz}; this is because, when $Z:=\mathbb{R}$ endowed with the
usual order, having $f:X\rightarrow\overline{\mathbb{R}}$ a proper function
and $\mathcal{A}:=\{(x,F_{f}(x))\mid x\in\operatorname*{dom}f\}$ (where
$F_{f}(x):=\{f(x)\}$ for $x\in\operatorname*{dom}f$, $F_{f}(x):=\emptyset$ for
$x\in X\backslash\operatorname*{dom}f$) endowed with the preorder
$(x,\{f(x)\})\preceq(x^{\prime},\{f(x^{\prime})\})$ defined by $f(x^{\prime
})\geq f(x)+d(x,x^{\prime})$ (or equivalently $F_{f}(x^{\prime})\subset
F_{f}(x)+d(x,x^{\prime})+\mathbb{R}_{+}$) when $x,x^{\prime}\in
\operatorname*{dom}f$, the conclusion of Theorem \ref{t-hz} is saying that for
every $x_{0}\in\operatorname*{dom}f$ there exists $\overline{x}\in
\operatorname*{dom}f$ such that $f(\overline{x})+d(x_{0},\overline{x})\leq
f(x_{0})$ and $f(\overline{x})<f(x)+d(x,\overline{x})$ for $x\in
X\backslash\{\overline{x}\}$, that is, the conclusions of one of the usual
variants of the Ekeland variational principle (EVP for short). One of the
hypotheses of EVP is the lower boundedness of $f$; this is equivalent to each
of the following conditions:%
\begin{align*}
\exists B  &  \subset\mathbb{R}\text{ bounded : }F_{f}(X)\subset
B+\mathbb{R}_{+},\\
\exists B  &  \subset\mathbb{R}\text{ bounded, }\forall x\in X:B\not \subset
F_{f}(x)+\mathbb{R}_{+},\\
\exists a  &  >0,~\exists\alpha\in\mathbb{R},~\forall t\in F_{f}(X):\alpha\leq
at.
\end{align*}
In the $(\varepsilon,\lambda)$ variants of EVP (with $\varepsilon,\lambda>0$)
one has a fixed $x_{0}\in X$ with $f(x_{0})\leq\inf f+\varepsilon$. When using
$F_{f}$ this condition is very close to
\[
\forall x\in X:F_{f}(x_{0})\not \subset F_{f}(x)+\varepsilon+\mathbb{R}%
_{+}=F_{f}(x)+\varepsilon\cdot1+\mathbb{R}_{+}.
\]

The natural extensions of these conditions for a nonempty subset $\mathcal{A}$
of $X\times2^{Y}$ (in particular for $\Gamma:X\rightrightarrows Y$ with
$\operatorname*{dom}\Gamma\neq\emptyset$ and $\mathcal{A}:=\{(x,\Gamma(x))\mid
x\in\operatorname*{dom}\Gamma\}$) and the nonempty subset $H\subset K$ are:
\begin{align}
\exists B &  \subset Y\text{ bounded :}Y_{\mathcal{A}}\subset
B+K,\label{r-bd1}\\
\exists B &  \subset Y\text{ bounded, }\forall A\in Y_{\mathcal{A}%
}:B\not \subset A+K,\label{r-bd2}\\
\exists y^{\ast} &  \in K^{+}\backslash\{0\},~\exists\alpha\in\mathbb{R}%
,~\forall y\in Y_{\mathcal{A}}:\alpha\leq y^{\ast}(y),\label{r-bd3}\\
\forall A &  \in\Pr\nolimits_{2^{Y}}(\mathcal{A}):A_{0}\not \subset
A+\varepsilon H+K,\label{r-bd4}%
\end{align}
where $\varepsilon>0$ and $A_{0}\in\Pr\nolimits_{2^{Y}}(\mathcal{A})$. In the
literature, one finds several boundedness notions wrt the convex cone
$K\subset Y$. Having the set $E\subseteq Y$, $E$ is (vector) $K$--bounded from
below if there exists $b\in Y$ such that $E\subset b+K;$ $E$ is quasi
$K$--bounded from below if there exists a bounded set $B\subset Y$ such that
$E\subset B+K;$ $E$ is $K$--bounded if for every neighborhood $U$ of $0\in Y$
there exists $\lambda>0$ such that $E\subset\lambda U+K$; $E$ is $K^{+}%
${--bounded} from below if $y^{\ast}(E)$ is bounded from below for every
$y^{\ast}\in K^{+}$.

In the following, the \textquotedblleft from below" part of the corresponding
expression will be dropped if, but we will keep the \textquotedblleft(vector)"
part in the first notion since it refers to boundedness from below wrt the
vector preorder generated by $K$ and in order to avoid confusion with $K$--boundedness.

Among these boundedness notions the implications below hold:
\begin{align}
E\text{ is (vector) K\text{--bounded} } &  \Rightarrow E\text{ is quasi
}K\text{--bounded}\nonumber\\
&  \Rightarrow E\text{ is }K\text{--bounded }\nonumber\\
&  \Rightarrow E\text{ is }K^{+}\text{--bounded;}\label{r-bound}%
\end{align}
moreover, if $\operatorname*{int}K\neq\emptyset$ we have that $E$ is
$K$--bounded $\Rightarrow$ $E$ is (vector) $K$--bounded.

The first implication is obvious. The second implication is stated in
{\cite[Lem.~1.3.2]{Luc89b}; its converse implication is true when }$Y$ is
normable as seen in \cite[Lem.~1.3.2]{Luc89b}, too. For the third implication
consider $y^{\ast}\in K^{+}$ and take $U:=\{y\in Y\mid\left\langle y,y^{\ast
}\right\rangle \geq-1\}$. Because $U\in\mathcal{N}_{Y}$, there exists
$\lambda>0$ such that $E\subset\lambda U+K$, and so
\[
y^{\ast}(E)\subset\lambda y^{\ast}(U)+y^{\ast}(K)\subset\lambda\cdot
\lbrack-1,\infty)+[0,\infty)=[-\lambda,\infty).
\]
Assume now that $\operatorname*{int}K\neq\emptyset$ and $E$ is $K$--bounded.
Take $k^{0}\in\operatorname*{int}K$, that is $K-k^{0}\in\mathcal{N}_{Y}$. Then
there exists $\lambda>0$ such that $E\subset\lambda(K-k^{0})+K=(-\lambda
k^{0})+K$, whence $E$ is (vector) $K$--bounded.

\begin{example}
\label{ex-bd1}It is known that the mapping $\left\Vert \cdot\right\Vert _{p}$
defined by $\left\Vert f\right\Vert _{p}:=\left(  \int_{(0,1)}\left\vert
f\right\vert ^{p}d\lambda\right)  ^{1/p}$ (where $\lambda$ is the Lebesgue
measure) is a quasinorm on $Y:=L^{p}(0,1)$ for $p\in(0,1)$. So $Y$ is a
topological vector space; it is also known that $Y^{\ast}=\{0\}$, see \cite[p.
158, Eq.~(9)]{KoetheVol1:69}. Taking $K:=\{f\in L^{p}(0,1)\mid f\geq0$
a.e.$\}$ we have that $K$ is a closed convex cone with $K^{+}=\{0\}$, and so
$Y$ is $K^{+}$--bounded, but clearly, $Y$ is not $K$--bounded.
\end{example}

For a subset $E$ of $Y$ we denote by $\operatorname*{cl}_{seq}E$ the set of
those $y\in Y$ such that there exists $(y_{n})_{n\geq1}\subset E$ with
$y_{n}\rightarrow y$; we say that $E$ is sequentially closed (seq-closed for
short) if $E=\operatorname*{cl}_{seq}E$. Moreover, we say that $E$ is
sequentially compact (seq-compact for short) if any sequence from $E$ has a
subsequence converging to an element of $E$. Of course, if the topology of $Y$
is metrizable, then $\operatorname*{cl}_{seq}E=\operatorname*{cl}E$, and $E$
is seq-compact if and only if $E$ is compact.

In Proposition \ref{p-bd} and Example \ref{ex-bd} (below) we mention several
relations among conditions (\ref{r-bd1})--(\ref{r-bd4}). First we give a
preliminary result.

\begin{lemma}
\label{p-bd0}Let $E\subset Y$ be nonempty, and set $\widetilde{E}%
:=E\cap\left(  -\mathbb{P}H-K\right)  $, where $\mathbb{P}:=\,]0,\infty[$. If
one of the following conditions holds:

\emph{(i) (a) }either $\widetilde{E}$ is quasi $K$--bounded and $0\notin
\operatorname*{cl}_{seq}(H+K)$, \emph{or} $\widetilde{E}$ is $K$--bounded and
$0\notin\operatorname*{cl}(H+K)$;

\emph{(ii)} there exists $z^{\ast}\in K^{+}$ such that $\inf z^{\ast
}(\widetilde{E})>-\infty$ and $\inf z^{\ast}(H)>0$;

\emph{(iii)} $\widetilde{E}$ is $K^{+}$--bounded and (\ref{zv-f5b}) is satisfied,

\noindent then%
\begin{equation}
\exists\varepsilon>0:E\cap(-\varepsilon H-K)=\emptyset. \label{r-bd6}%
\end{equation}

Moreover, if $(x_{0},A_{0})\in\mathcal{A}\subset X\times2^{Y}$, $y_{0}\in
A_{0}$ and (\ref{r-bd6}) holds for ${E}:=(Y_{\mathcal{A}}-y_{0})$, then
(\ref{r-bd4}) holds with $\varepsilon$ provided by (\ref{r-bd6}).
\end{lemma}

\begin{proof}
We prove (i)--(iii) by contradiction. So, assume that (\ref{r-bd6}) does not
hold, and so for each $n\in\mathbb{N}\backslash\{0\}$ there exists
\begin{equation}
y_{n}\in E\cap(-nH-K)\subset E\cap\left(  -\mathbb{P}(H+K)\right)
=\widetilde{E}; \label{r-bd8}%
\end{equation}
hence there exist $h_{n}\in H$ and $k_{n}\in K$ such that $y_{n}=-nh_{n}%
-k_{n}$ (for $n\geq1$). It follows that $\beta_{z^{\ast}}:=\inf z^{\ast
}(\widetilde{E})<+\infty$ for $z^{\ast}\in K^{+}$. Moreover,%
\begin{equation}
\beta_{z^{\ast}}\leq z^{\ast}(y_{n})=-n\cdot z^{\ast}(h_{n})-z^{\ast}%
(k_{n})\leq-n\cdot z^{\ast}(h_{n})\quad\forall n\geq1,~\forall z^{\ast}\in
K^{+}. \label{r-bd0}%
\end{equation}

(i) (a) Let $\widetilde{E}\subset B+K$ with $B\subset Y$ a bounded set. Since
$y_{n}\in\widetilde{E}$, there exist also $b_{n}\in B$ and $k_{n}^{\prime}\in
K$ such that $y_{n}=b_{n}+k_{n}^{\prime}$ for $n\geq1$. It follows that
$-\frac{1}{n}b_{n}=h_{n}+\frac{1}{n}(k_{n}+k_{n}^{\prime})\in H+K$ for
$n\geq1$. Since $(b_{n})$ is bounded, we get the contradiction $0\in
\operatorname*{cl}_{seq}(H+K)$.

(b) Because $0\notin\operatorname*{cl}(H+K)$, there exists a
balanced $U\in\mathcal{N}_{Y}$ such that $U\cap(H+K)=\emptyset $.
Since $\widetilde{E}$ is $K$--bounded, there exists $\lambda>0$ such
that $\widetilde{E}\subset\lambda U+K$. Since
$y_{n}\in\widetilde{E}$, there exists
$u_{n}\in U$ and $k_{n}^{\prime}\in K$ such that $y_{n}=\lambda u_{n}%
+k_{n}^{\prime}$ for $n\geq1$. It follows that $-\frac{\lambda}{n}u_{n}%
=h_{n}+\frac{1}{n}(k_{n}+k_{n}^{\prime})\in H+K$ for $n\geq1$. Taking
$n\in\mathbb{N}^{\ast}$ such that $n\geq\lambda$, we obtain the contradiction
$-\frac{\lambda}{n}u_{n}\in U\cap(H+K)\neq\emptyset$.

(ii) For $z^{\ast}$ provided by our assumption we have that
$\alpha_{z^{\ast}}:=\inf z^{\ast }(H)>0$ and
$\beta_{z^{\ast}}\in\mathbb{R}$. Using (\ref{r-bd0}) we obtain that
$\beta_{z^{\ast}}\leq-n\alpha_{z^{\ast}}$ for $n\geq1$. Taking the
limit for $n\rightarrow\infty$ we get the contradiction
$\beta_{z^{\ast}}=-\infty$.

(iii) By (\ref{zv-f5b}), there exists $z^{\ast}\in K^{+}$ such that
$\limsup
z^{\ast}(h_{n})>0$. From (\ref{r-bd0}) we have that $z^{\ast}(h_{n}%
)\leq-n^{-1}\beta_{ z^{\ast}}$ for $n\geq1$, whence the contradiction $\limsup
z^{\ast}(h_{n})\leq0$.

For the last assertion just observe that, for
$E:=(Y_{\mathcal{A}}-y_{0})$, (\ref{r-bd6}) is saying that there
exists $\varepsilon>0$ such that $y_{0}\notin
Y_{\mathcal{A}}+\varepsilon H+K$ which clearly implies
(\ref{r-bd4}).
\end{proof}

\medskip Using Lemma \ref{p-bd0}~(i) one gets \cite[Prop.\ 2.2]{Qiu:12} taking
$E:=\Gamma(X)-y$ with $y\in Y$; there $H=\{k_{0}\}\subset K\backslash\{0\}$,
$K$ being a pointed closed convex cone. Lemma \ref{p-bd0}~(ii) is obtained in
\cite[Prop.\ 5.2]{Qiu:16b}; there $E:=f(X)-f(x_{0})$ for some function $f:X\to
Y$.

\begin{proposition}
\label{p-bd}Let $\mathcal{A}\subset X\times2^{Y}$ be such that $Y_{\mathcal{A}%
} \neq\emptyset$.

\emph{(i)} Assume that either \emph{(a)} $0\notin\operatorname*{cl}%
\nolimits_{seq}(H+K)$ and $Y_{\mathcal{A}}$ is quasi $K$--bounded, or
\emph{(b)}~$0\notin\operatorname*{cl}(H+K)$ and $Y_{\mathcal{A}}$ is
{$K$--bounded}; then
\begin{equation}
\forall y\in Y,~\exists\varepsilon>0:Y_{\mathcal{A}}\cap(y-\varepsilon
H-K)=\emptyset. \label{r-bd5}%
\end{equation}

\emph{(ii)} Assume that either \emph{(a)} there exists $z^{\ast}\in K^{+}$
such that $\inf z^{\ast}(Y_{\mathcal{A}})>-\infty$ and $\inf z^{\ast}(H)>0$,
or \emph{(b)}$~Y_{\mathcal{A}}$ is $K^{+}$--bounded and (\ref{zv-f5b}) is
satisfied. Then (\ref{r-bd5}) holds.

\emph{(iii)} Assume that $H$ is seq-compact. If there exists a convex cone
$C\subset Y$ and a bounded set $B\subset Y$ such that $K\subset C$,
$H\subset\operatorname*{int}C$ and $B\not \subset A+C$ for every $A\in
\Pr_{2^{Y}}(\mathcal{A})$, then (\ref{r-bd5}) holds; in particular, if
$H\subset\operatorname*{int}K$ and (\ref{r-bd2}) is verified then
(\ref{r-bd5}) holds.

\emph{(iv)} Assume that (\ref{r-bd5}) is verified. Then (\ref{r-bd2}) holds,
and for every nonempty set $A_{0}\in\Pr\nolimits_{2^{Y}}(\mathcal{A})$ there
exists $\varepsilon>0$ such that (\ref{r-bd4}) holds.
\end{proposition}

\begin{proof}
(i) (a) Take $y\in Y$ and set $E:=Y_{\mathcal{A}}-y$. Then $E$ is quasi
$K$--bounded and $\widetilde{E}:=E\cap\left(  -\mathbb{P}(H+K)\right)  \subset
E$, and so $\widetilde{E}$ is quasi $K$--bounded. By Lemma \ref{p-bd0}~(i),
there exists $\varepsilon>0$ such that $Y_{\mathcal{A}}-y=E\subset-\varepsilon
H-K$, whence the desired conclusion holds. The proof of case (b) is similar.

(ii) Using now Lemma \ref{p-bd0}~(ii) and (iii), the same approach as in (i)
shows that (\ref{r-bd5}) holds when (a) or (b) is verified, respectively.

(iii) Assume that (\ref{r-bd5}) does not hold. Then there exists $y_{0}\in Y$
such that for every $n\in\mathbb{N}^{\ast}$, there exists $A_{n}\in
\Pr\nolimits_{2^{Y}}(\mathcal{A})$ with $A_{n}\cap(y_{0}-nH-K)\neq\emptyset;$
take $y_{n}\in A_{n}\cap(y_{0}-nH-K)$. Since $B\not \subset A_{n}+C$,
$B\not \subset y_{n}+C$, and so there exists $b_{n}\in B$ such that
$b_{n}\notin y_{n}+C$. It follows that $b_{n}-y_{n}\notin C$, whence
$b_{n}-y_{n}\in Y\backslash C\subset Y\backslash\operatorname*{int}C$. From
our choice of $y_{n}$, there exists $h_{n}\in H$ such that $y_{n}-y_{0}%
+nh_{n}\in-K\subset-C$. It follows that
\[
b_{n}-y_{0}+nh_{n}=(b_{n}-y_{n})+(y_{n}-y_{0}+nh_{n})\in(Y\backslash
\operatorname*{int}C)-C\subset Y\backslash\operatorname*{int}C.
\]
Since $Y\backslash\operatorname*{int}C$ is a cone, we obtain that $\frac{1}%
{n}(b_{n}-y_{0})+h_{n}\in Y\backslash\operatorname*{int}C$. Because $H$ is
seq-compact, there exists a strictly increasing sequence $(n_{p})_{p\geq
1}\subset\mathbb{N}^{\ast}$ such that $h_{n_{p}}\rightarrow h\in H$. It
follows that $\frac{1}{n_{p}}(b_{n_{p}}-y_{0})+h_{n_{p}}\rightarrow
h\in\operatorname*{cl}_{seq}\left(  Y\backslash\operatorname*{int}C\right)
=Y\backslash\operatorname*{int}C$, contradicting the fact that $H\subset
\operatorname*{int}C$.

(iv) For the first assertion, taking $y:=0$ in (\ref{r-bd5}), there exists
$\varepsilon>0$ such that $Y_{\mathcal{A}}\cap(-\varepsilon H-K)=\emptyset$.
Hence (\ref{r-bd2}) holds for $B:=\{-\varepsilon h^{0}\}$ with $h^{0}\in H$.

For the second assertion, take $A_{0}\in\Pr\nolimits_{2^{Y}}(\mathcal{A})$
with $A_{0}\neq\emptyset$. Consider $y_{0}\in A_{0}$. Since (\ref{r-bd5})
holds, there exists $\varepsilon>0$ such that $Y_{\mathcal{A}}\cap
(y_{0}-\varepsilon H-K)=\emptyset$, that is $y_{0}\notin Y_{\mathcal{A}%
}+\varepsilon H+K$. Hence $A_{0}\not \subset Y_{\mathcal{A}}+\varepsilon H+K$,
and so $A_{0}\not \subset A+\varepsilon H+K$ for any $A\in\Pr\nolimits_{2^{Y}%
}(\mathcal{A})$. The proof is complete.
\end{proof}

\medskip

Some relations among the conditions on $H$ used in Proposition \ref{p-bd} are
provided in the next result.

\begin{lemma}
\label{lem0}\emph{(i)} If $\emptyset\neq S,T\subset Y$ and $S$ is seq-compact
then $\operatorname*{cl}_{seq}(S+T)=S+\operatorname*{cl}_{seq}T$; in
particular, if $H$ is seq-compact, then $0\not \in \operatorname*{cl}%
_{seq}(H+K)$ if and only if $H\cap(-\operatorname*{cl}_{seq}K)=\emptyset$.

\emph{(ii)} The following implications hold:%
\begin{equation}
\left[  \exists z^{\ast}\in K^{+}:\inf z^{\ast}(H)>0\right]  \Rightarrow
(\ref{zv-f5b})\Rightarrow0\not \in \operatorname*{cl}\nolimits_{seq}(H+K).
\label{zv-f5c}%
\end{equation}

\emph{(iii)} Assume that $H$ is seq-compact and $C\subset Y$ is a proper
convex cone such that $K\subset C$ and $H\subset\operatorname*{int}C$. Then
$0\not \in \operatorname*{cl}_{seq}(H+K)$.
\end{lemma}

\begin{proof}
(i) The inclusion $S+\operatorname*{cl}_{seq}T\subset\operatorname*{cl}
_{seq}(S+T)$ is obvious. Assume that $S$ is seq-compact and take
$y\in\operatorname*{cl}_{seq}(S+T)$. Then there exist $(u_{n})_{n\geq1}\subset
S$ and $(v_{n})_{n\geq1}\subset T$ such that $u_{n}+v_{n}\rightarrow y$.
Because $S$ is seq-compact, there exists a strictly increasing sequence
$(n_{p})_{p\geq1}\subset\mathbb{N}^{\ast}$ such that $u_{n_{p}}\rightarrow
u\in S$. It follows that $v_{n_{p}}=(u_{n_{p}}+u_{n_{p}})-u_{n_{p}}\rightarrow
v:=y-u$, and so $y=u+v\in S+\operatorname*{cl}_{seq}T$.

(ii) The first implication in (\ref{zv-f5c}) is obvious. Assume that
$0\in\operatorname*{cl}_{seq}(H+K)$ and take $z^{\ast}\in K^{+}$. Then there
exists $(h_{n})_{n\geq1}\subset H$ and $(k_{n})_{n\geq1}\subset K$ such that
$y_{n}:=h_{n}+k_{n}\rightarrow0$. Since $z^{\ast}(y_{n})=z^{\ast}%
(h_{n})+z^{\ast}(k_{n})\geq z^{\ast}(h_{n})\geq0$, it follows that $z^{\ast
}(h_{n})\rightarrow0$. This shows that (\ref{zv-f5b}) does not hold. Hence,
the second implication (\ref{zv-f5c}) holds, too.

(iii) Clearly, $\operatorname*{cl}_{seq}K\subset\operatorname*{cl}C$, and so,
using (i) and the properness of $C$,
\[
\operatorname*{cl}\nolimits_{seq}(H+K)=H+\operatorname*{cl}\nolimits_{seq}K
\subset\operatorname*{int} C+\operatorname*{cl}C=\operatorname*{int}C.
\]
Since $0\notin\operatorname*{int}C$ ($C$ being proper), $0\notin
\operatorname*{cl}\nolimits_{seq}(H+K)$.
\end{proof}

\medskip

The example below shows that some of the reversed implications in Proposition
\ref{p-bd} are not valid.

\begin{example}
\label{ex-bd}Let $Y:=\mathbb{R}^{2}$ be endowed with the Euclidean norm,
$K:=\mathbb{R}_{+}^{2}$, $H:=\{k^{0}\}$ with $k^{0}:=(1,0)\in K$ and
$E:=\{0\}\times\mathbb{R}_{-}$, where $\mathbb{R}_{-}:=-\mathbb{R}_{+}$. It is
clear that $H$ is compact, $E$ is not $K^{+}$--bounded (for example
$\varphi:\mathbb{R}^{2}\rightarrow\mathbb{R}$ defined by $\varphi(y_{1}%
,y_{2}):=y_{2}$ is in $K^{+}$ and $\inf\varphi(E)=-\infty$), and so, by
(\ref{r-bound}), $E$ is not quasi $K$--bounded. However, (\ref{r-bd5}) holds
for $Y_{\mathcal{A}}\subset E$. Indeed, for $y=(y_{1},y_{2})\in\mathbb{R}^{2}%
$, take $\varepsilon:=1+\max\{0,y_{1}\}$; then $E\cap(y-\varepsilon
k^{0}-K)=\emptyset$.
\end{example}

The following result will play an important role in the next sections.

\begin{proposition}
\label{lem4}Assume that $\left(  (x_{n},A_{n})\right)  _{n\geq1}\subset
X\times2^{Y}$ is a $\preceq_{H}$--decreasing sequence with $A_{1}\neq
\emptyset$. Set $E:=\cup_{n\geq1}A_{n}$, and consider the following assertions:

\emph{(i)} either \emph{(a)} $0\notin\operatorname*{cl}\nolimits_{seq}(H+K)$
and $E$ is quasi $K$--bounded, or \emph{(b)} $0\notin\operatorname*{cl}(H+K)$
and $E$ is $K$--bounded;

\emph{(ii)} there exists $z^{\ast}\in K^{+}$ such that $\inf z^{\ast
}(E)>-\infty$ and $\inf z^{\ast}(H)>0$;

\emph{(iii)} the set $E$ is $K^{+}$--bounded and condition (\ref{zv-f5b}) is satisfied;

\emph{(iv) }$H$ is seq-compact, and there exist a bounded set $B\subset Y$ and
a convex cone $C$ such that $K\subset C$, $H\subset\operatorname*{int}C$ and
$B\not \subset A_{n}+C$ for $n\geq1;$

\emph{(v)} there exist $\alpha>0$ and an infinite set $P\subset\mathbb{N}
^{\ast}$ such that $A_{1}\not \subset (A_{n}+\alpha H+K)$ for all $n\in P$.

Then \emph{[(i) }$\vee$ \emph{(ii) }$\vee$\emph{ (iii) }$\vee$\emph{
(iv)]} $\Rightarrow$ \emph{(v)}
$\Rightarrow\sum_{n\geq1}d(x_{n},x_{n+1})\leq\alpha;$ in particular,
if at least one of the assertions \emph{(i)--(v)} holds, then
$(x_{n})_{n\geq1}$ is a Cauchy sequence.
\end{proposition}

\begin{proof}
Assume that (i), or (ii), or (iii), or (iv) holds. Taking $\mathcal{A}%
:=\{(x_{n},A_{n})\mid n\in\mathbb{N}\backslash\{0\}\}$ and using
Proposition \ref{p-bd} (i), (ii)(a), (ii)(b) or (iii), respectively,
we obtain that (\ref{r-bd5}) holds. Applying now Proposition
\ref{p-bd} (iv), we obtain that (v) holds.

Assume that (v) holds. Since $(x_{n+1},A_{n+1})\preceq_{H}^{1}(x_{n},A_{n})$
for $n\geq1$, we have that
\begin{equation}
A_{n}\subset A_{n+1}+d(x_{n},x_{n+1})H+K\quad\forall n\geq1. \label{zv-eq5}%
\end{equation}
Using (\ref{r-hk}) and the inclusions in (\ref{zv-eq5}), we obtain that
\begin{equation}
A_{1}\subset A_{n+1}+\mu_{n}H+K\quad\forall n\geq1, \label{zv-eq6}%
\end{equation}
where $\mu_{n}:=\sum_{l=1}^{n}d(x_{l},x_{l+1})$. Since $A_{1}\not \subset
A_{n+1}+\alpha H+K$ for $n+1\in P$, we have that $\mu_{n}<\alpha$; otherwise
(that is $\mu_{n}\geq\alpha$), from (\ref{zv-eq6}) we get the contradiction
\[
A_{1}\subset A_{n+1}+\mu_{n}H+K\subset A_{n+1}+(\mu_{n}-\alpha)H+\alpha
H+K\subset A_{n+1}+\alpha H+K.
\]
The set $P$ being infinite, $\mu_{n}<\alpha$ for every $n\geq1$, and so
$\sum_{n\geq1}d(x_{n},x_{n+1})=\lim\mu_{n}\leq\alpha$.
\end{proof}

\begin{corollary}
\label{c-bd}Let $\mathcal{A}\subset X\times2^{Y}$ be such that $A\neq
\emptyset$ for every $A\in\Pr_{2^{Y}}(\mathcal{A})$. Then $(\mathcal{A}
,\preceq_{H})$ verifies condition \emph{(Ca)} whenever one of the following
assertions holds:

\emph{(i)} either \emph{(a)} $0\notin\operatorname*{cl}\nolimits_{seq}(H+K)$
and $Y_{\mathcal{A}}$ is quasi $K$--bounded, or \emph{(b)} $0\notin
\operatorname*{cl}(H+K)$ and $Y_{\mathcal{A}}$ is $K$--bounded;

\emph{(ii)} there exists $z^{\ast}\in K^{+}$ such that $\inf z^{\ast
}(Y_{\mathcal{A}})>-\infty$ and $\inf z^{\ast}(H)>0$;

\emph{(iii)} $Y_{\mathcal{A}}$ is $K^{+}$--bounded and condition
(\ref{zv-f5b}) is satisfied;

\emph{(iv) }$H$ is seq-compact, and there exist a bounded set $B\subset Y$ and
a convex cone $C$ such that $K\subset C$, $H\subset\operatorname*{int}C$ and
$B\not \subset A+C$ for every $A\in\Pr_{2^{Y}}(\mathcal{A})$.
\end{corollary}

\begin{proof}
It is an immediate consequence of Propositions \ref{p-bd} and \ref{lem4}.
\end{proof}

\medskip

In the next section we shall use also cs-complete sets. Recall that a set
$E\subseteq Y$ is \emph{cs-complete}
(see \cite[p.\ 9]{Zal:02}) if for all sequences $(\lambda_{n})_{n\geq
1}\subseteq\mathbb{R}_{+}$, $(y_{n})_{n\geq1}\subseteq E$ such that
$\sum_{n\geq1}\lambda_{n}=1$ and the sequence $\left(  \sum_{m=1}^{n}%
\lambda_{m}y_{m}\right)  _{n\geq1}$ is Cauchy, the series $\sum_{n\geq
1}\lambda_{n}y_{n}$ (called convex series with elements of $E$) is convergent
and its sum belongs to $E$. One says that $E\subseteq Y$ is \emph{cs-closed}
if the sum of any convergent convex series with elements from $E$ belongs to
$E$. Of course, any cs-complete set is cs-closed; if $Y$ is (sequentially)
complete then the converse is true. Notice also that any cs-closed set is
convex. Moreover, it is worth observing that the closed convex subsets of
topological vector spaces are cs-closed, as well as the open convex subsets of
separated locally convex spaces; furthermore, all the convex subsets of finite
dimensional normed spaces are cs-closed (hence cs-complete).

Notice that $E\subset Y$ is cs-complete if and only if for all sequences
$(\lambda_{n})_{n\geq1}\subseteq\mathbb{R}_{+}$, $(y_{n})_{n\geq1}\subseteq E$
such that $\lambda:=\sum_{n\geq1}\lambda_{n}\in\mathbb{R}_{+}^{\ast}$ and the
sequence $\left(  \sum_{m=1}^{n}\lambda_{m}y_{m}\right)  _{n\geq1}$ is Cauchy,
the series $\sum_{n\geq1}\lambda_{n}y_{n}$ converges to some $y\in Y$ and
$\lambda^{-1}y\in E$. Similarly, $E$ is cs-closed if and only if for all
sequences $(\lambda_{n})_{n\geq1}\subseteq\mathbb{R}_{+}$, $(y_{n})_{n\geq
1}\subseteq E$ such that $\lambda:=\sum_{n\geq1}\lambda_{n}\in\mathbb{R}%
_{+}^{\ast}$ and $y:=\sum_{n\geq1}\lambda_{n}y_{n}\in Y$, one has
$\lambda^{-1}y\in E$. Note that in \cite[Def.\ 2.1.4]{PerBon:87} it is said
that $E$ is $\sigma$-convex or cs-compact if any convex series with elements
from $E$ is convergent and its sum belongs to $E$. It is quite easy to prove
that $E\subset Y$ is $\sigma$-convex if and only if $E$ is cs-complete and
bounded; this assertion is formulated in \cite[p.~19]{Qiu:16b} (see also
\cite[p.~921]{Qiu:14}, where it is asserted that $E$ is $\sigma$-convex if and
only if $E$ is cs-complete).

\begin{proposition}
\label{p-csc}Assume that $E\subset Y$ is a nonempty seq-compact convex set.
Then $E$ is cs-complete.
\end{proposition}

\begin{proof}
Since $E$ is seq-compact, $E$ is bounded. Take $(\lambda_{n})_{n\geq1}%
\subset\mathbb{R}_{+}^{\ast}$ with $\sum_{n\geq1}\lambda_{n}=1$ and
$(v_{n})_{n\geq1}\subset E$. For this, fix $V\subseteq Y$ a balanced
neighborhood of $0$. Because $E$ is bounded, there exists $\alpha>0$ such that
$E\subseteq\alpha V$. Since the series $\sum_{n\geq1}\lambda_{n}$ is
convergent there exists $n_{0}\geq1$ such that $\sum_{k=n}^{n+p}\lambda
_{k}\leq\alpha^{-1}$ for all $n,p\in\mathbb{N}$ with $n\geq n_{0}$. Since $E$
is convex, for such $n,p$ there exists $v_{n,p}\in E$ with%
\[
\sum\nolimits_{k=n}^{n+p}\lambda_{k}v_{k}=\left(  \sum\nolimits_{k=n}%
^{n+p}\lambda_{k}\right)  v_{n,p}\in\lbrack0,\alpha^{-1}]E\subseteq
\lbrack0,\alpha^{-1}]\alpha V=V;
\]
hence $(y_{n})_{n\geq1}$ is a Cauchy sequence, where $y_{n}:=\sum_{m=1}%
^{n}\lambda_{m}v_{m}$ for $n\geq1$. Set also $\mu_{n}:=\sum_{m=1}^{n}%
\lambda_{m}\in\lbrack0,1]$. Since $\mu_{n}\rightarrow1$, we may (and do)
assume that $\mu_{n}>0$ for $n\geq1$. Since $E$ is convex, $v_{n}^{\prime
}:=\mu_{n}^{-1}y_{n}\in E$, while because $E$ is seq-compact, $(v_{n}^{\prime
})$ has a subnet converging to $v\in E$. Because $\mu_{n}\rightarrow1$, the
corresponding subnet of $(y_{n})$ converges to $v$. Since $(y_{n})$ is a
Cauchy sequence, $y_{n}\rightarrow v$. Hence $E$ is cs-complete.
\end{proof}

\medskip

Clearly, the converse of the preceding result is not true in general; for
example, any nonempty open convex subset of a Banach space is cs-complete
(but, clearly, is not compact).


\section{Ekeland's Variational Principles of Ha--Hamel--L\"{o}hne's Type}


Throughout this section $X$, $Y$, $K$, $F$ are as in Section \ref{sec-vz-10.1}
(if not stated otherwise explicitly), that is, $(X,d)$ is a metric space, $Y$
is a separated topological vector space, $K\subset Y$ is a proper convex cone,
and $F:X\times X\rightrightarrows K$ verifies conditions (F1) and (F2). On
$X\times2^{Y}$ we consider the preorder $\preceq_{F}$, as well as
$\preceq_{F,z^{\ast}}$ for $z^{\ast}\in K^{+}$, defined in (\ref{zv-eq1-b})
and (\ref{vz-eq5}), respectively. Taking into account that $(x,A)\preceq
_{F}(x^{\prime},\emptyset)$ for all $x,x^{\prime}\in X$ and $A\in2^{Y}$, in
the sequel we assume that

\begin{description}
\item[(H)] $\mathcal{A}\subseteq X\times2^{Y}$ is nonempty and $A\neq
\emptyset$ for every $A\in{\Pr}_{2^{Y}}(\mathcal{A})$.
\end{description}

Hence
\[
Y_{\mathcal{A}}:=\bigcup\big\{A\mid A\in{\Pr}_{2^{Y}}(\mathcal{A}%
)\big\}\neq\emptyset.
\]
An important example of set $\mathcal{A}\subseteq X\times2^{Y}$ which will be
considered often in the sequel is
\[
\mathcal{A}_{\Gamma}:=\big\{\big((x,\Gamma(x))\big)\mid x\in
\operatorname*{dom}\Gamma\big\},
\]
where $\Gamma:X\rightrightarrows Y$ with $\operatorname*{dom}\Gamma
\neq\emptyset;$ of course, $Y_{\mathcal{A}_{\Gamma}}=\Gamma(X)$. Ha
\cite{Ha05} established an EVP type result on a complete metric space $X$ for
a set-valued function $\Gamma:X\rightrightarrows Y$ which uses the relation
$\leq_{K}^{l}$ (see Section \ref{sec-vz-10.1}) for comparing the values of
$\Gamma$; Hamel \cite{Ham:05} and Hamel--L\"{o}hne \cite{HamLoh:06}
established more general results for subsets $\mathcal{A}\subseteq
X\times2^{Y}$ even for a uniform space $X$. Corollary \ref{zv-t4-Ha} below is
an extension of this type of results.

The first part of the next result is a translation of Theorem \ref{t-hz} to
the present context. For getting the second part we use Lemma \ref{lem-fi} and
Theorem \ref{ex-BB}.

\begin{theorem}
\label{c-hzH}Assume that $(\mathcal{A},\preceq_{F})$ verifies condition
\emph{(C0)}. Then:

\emph{(i)} for every $(x,A)\in\mathcal{A}$ there exists $(\overline
{x},\overline{A})\in\mathcal{A}$ such that $(\overline{x},\overline{A}
)\preceq_{F}(x,A)$ and $\mathcal{A}\ni(x^{\prime},A^{\prime})\preceq
_{F}(\overline{x},\overline{A})$ implies $x^{\prime}=\overline{x};$

\emph{(ii)} assume that $z^{\ast}\in K^{+}$ is such that $\inf
z^{\ast }(F(x,x^{\prime}))>0$ for $x,x^{\prime}\in X$ with $x\neq
x^{\prime}$ and $\inf z^{\ast }(A)>-\infty$ for
$A\in\Pr_{2^{Y}}(\mathcal{A})$; then for every $(x,A)\in\mathcal{A}$
there exists $(\overline{x},\overline{A} )\in\mathcal{A}$ minimal
wrt $\preceq_{F,z^{\ast}}$ such that $(\overline
{x},\overline{A})\preceq_{F,z^{\ast}}(x,A)$ and $\mathcal{A}
\ni(x^{\prime },A^{\prime})\preceq_{F}(\overline{x},\overline{A})$
implies $x^{\prime }=\overline{x}$.
\end{theorem}

\begin{proof}
(ii) Take $z^{\ast}\in K^{+}$ provided by the hypothesis. It is clear that
$(W,\preccurlyeq):=(\mathcal{A},\preceq_{F})$ is a preordered set. Moreover,
using the estimate in (\ref{zv-eq3b}) we obtain that
\[
\phi:\mathcal{A}\rightarrow\overline{\mathbb{R}},\quad\phi(x,A):=\inf z^{\ast
}(A),
\]
is $\preceq_{F}$--increasing and real-valued. Clearly, $\preccurlyeq_{\phi}$
from (\ref{r-lefi}) is nothing else than $\preceq_{F,z^{\ast}}$. Because
$(\mathcal{A},\preceq_{F})$ verifies condition (C0), using Proposition
\ref{rel-c} and Lemma \ref{lem-fi} (iii), $(W,\preccurlyeq_{\phi})$ verifies
condition (Ab).

Consider $(x,A)\in\mathcal{A}$ $(=W);$ since $\phi$ is strictly $\preccurlyeq
_{\phi}$--increasing, applying Theorem \ref{ex-BB}~(ii) for $(W,\preccurlyeq
_{\phi})$, we get $(\overline{x},\overline{A})\in\mathcal{A}$ minimal wrt
$\preceq_{\phi}=\preceq_{F,z^{\ast}}$ such that $(\overline{x},\overline
{A})\preceq_{F,z^{\ast}}(x,A)$. Take $\mathcal{A}\ni(x^{\prime},A^{\prime
})\preceq_{F}(\overline{x},\overline{A})$, that is $W\ni(x^{\prime},A^{\prime
})\preceq_{\phi}(\overline{x},\overline{A})$. Using Lemma \ref{lem-fi} (ii) we
obtain that $\inf z^{\ast}(A^{\prime})=\inf z^{\ast}(\overline{A})$. Assuming
that $x^{\prime}\neq\overline{x}$, we get the contradiction $\inf z^{\ast
}(\overline{A})\geq\inf z^{\ast}(A^{\prime})+\inf z^{\ast}\left(  F(x^{\prime
},\overline{x})\right)  >\inf z^{\ast}(A^{\prime})$ using (\ref{zv-eq3}). The
proof is complete.
\end{proof}

\begin{remark}
\label{zv-rem-rel2}{\rm Note that, for having the conclusions of
Theorem \ref{c-hzH} (i) or (ii) only for a given $(x,
A)\in\mathcal{A}$, it is sufficient to assume that (C0) is verified
by the sets
\begin{align*}
\mathcal{A}_{F}(x, A)  &  :=\{(x^{\prime}, A^{\prime})\in\mathcal{A}%
\mid(x^{\prime},A^{\prime})\preceq_{F}(x, A)\},\\
\mathcal{A}_{F,z^{\ast}}(x, A)  &  :=\{(x^{\prime}, A^{\prime})\in
\mathcal{A}\mid(x^{\prime}, A^{\prime})\preceq_{F,z^{\ast}}(x, A)\},
\end{align*}
respectively.}
\end{remark}

\begin{remark}
\label{rem-HL}{\rm Taking
$F(x,x^{\prime}):=\{d(x,x^{\prime})k^{0}\}$ with $k^{0}\in
K\backslash(-\operatorname*{cl}K)$ in Theorem \ref{c-hzH} (using
also Remark \ref{zv-rem-rel2}) one obtains \cite[Th.\
5.1]{HamLoh:06} in the case $X$ is a metric space. Indeed, on
$\mathcal{A}_{0}$, if (M2) or (M2') is verified, then (i) or (iv)
from Proposition \ref{lem4} holds, respectively, and so any
$\preceq_{F}$--decreasing sequence in $\mathcal{A}_{0}$ is Cauchy.
This together with (M3) shows that (C0) holds. In a similar way
\cite[Th.\ 6.1]{HamLoh:06} can be obtained.}
\end{remark}

In the sequel, the conditions (Cx) or (Cxx) will refer to $(\mathcal{A}%
,\preceq_{F})$. Note that condition (C'1) for $(\mathcal{A},\preceq_{F})$
corresponds to condition (H1) in \cite{TamZal:11} and \cite{KhaTamZal:15b}.

\begin{remark}
\label{rem-dyn2}{\rm Taking $X$, $Y$, $K$, $H$ and $\Gamma
:X\rightrightarrows Y$ defined in Example \ref{ex1} below, $(\mathcal{A}%
_{\Gamma},\preceq_{F})$ verifies conditions (C'1) and (C1) for $F:=F_{H}$, but
not (C0). Moreover, the conclusion of Theorem \ref{c-hzH}~(i) does not hold.
This shows that, in order to have the conclusion of Theorem \ref{c-hzH} we
need supplementary conditions besides (C'1) or (C1).}
\end{remark}

\begin{example}
\label{ex1}Let $X:=\mathbb{R}$ and $Y:=\mathbb{R}^{2}$ be endowed with their
usual norms, $K:=\mathbb{R}\times\mathbb{R}_{+}$, $H:=\{(y_{1},y_{2}
)\in\mathbb{R}_{+}^{2}\mid y_{1}y_{2}\geq1\}$, and $\Gamma:X\rightrightarrows
Y$, $\Gamma(x):=\{(x,e^{x})\}$. It is clear that $H$ is a closed convex subset
of $K\backslash\{(0,0)\}$ and $K+\varepsilon H=\operatorname*{int}
K=\mathbb{R}\times\mathbb{R}_{+}^{\ast}$ for $\varepsilon>0$, where
$\mathbb{R}_{+}^{\ast}:=\mathbb{R}_{+}\backslash\{0\}$. One has $\Gamma
(x)+K=\mathbb{R}\times\lbrack e^{x},\infty)$, and so, for $x,x^{\prime}\in X$
and $\alpha>0$,
\begin{equation}
\Gamma(x)\preceq_{H}\Gamma(x^{\prime})\Leftrightarrow x\leq x^{\prime
}\Leftrightarrow\Gamma(x)\leq_{K}^{l}\Gamma(x^{\prime}),\quad\Gamma(x^{\prime
})\subset\Gamma(x)+\alpha H+K\Leftrightarrow x<x^{\prime}; \label{r-ex1}%
\end{equation}
moreover, $\Gamma(X)=\{(x,e^{x})\mid x\in\mathbb{R}\}\subset K$, which shows
that $\Gamma(X)$ is (vector) $K$--bounded. So, for the sequence $(x_{n}%
)_{n\geq1}\subset X$ with $x_{n}\rightarrow x$ and $\Gamma(x_{n})\subset
\Gamma(x_{n+1})+K$ (that is $\Gamma(x_{n+1})\leq_{K}^{l}\Gamma(x_{n})$) for
$n\geq1$, we have that $x_{n+1}\leq x_{n}$, and so $x\leq x_{n}$ for $n\geq1$,
whence $\Gamma(x_{n})\subset\Gamma(x)+K$ for $n\geq1$. This shows that
\emph{(C'1)} and \emph{(C1)} are verified; however, taking $x_{n}:=-n$ (for
$n\geq1$) it is clear that \emph{(C0)} is not verified.
\end{example}

\begin{remark}
\label{p-dyn}{\rm Let $\Gamma:X\rightrightarrows Y$ have nonempty
domain. Set $x\preceq u$ if $(x,\Gamma(x))\preceq_{F}(u,\Gamma(u))$.
Using Remark \ref{zv-rem7}, $(\mathcal{A}_{\Gamma},\preceq_{F})$
verifies condition (C'1) if and only if $S(u):=\{x\in X\mid x\preceq
u\}$ is $\preceq$--lower closed for every $u\in X$, if and only if
$(X,d,\preceq)$ verifies (A'1). This shows that condition (C'1)
extends the dynamic closedness of a set-valued mapping as
defined in \cite{Qiu:14} and elsewhere. Also notice that $(\mathcal{A}%
_{\Gamma},\preceq_{F})$ verifies condition (Ca) if and only if $(X,d,\preceq)$
verifies (Aa).}
\end{remark}

\begin{lemma}
\label{lem5}Assume that $F$ verifies condition \emph{(F3)}, and $z_{F}^{\ast
}(Y_{\mathcal{A}})$ is bounded from below, where $z_{F}^{\ast}$ is provided by
\emph{(F3)}. Then $(\mathcal{A},\preceq_{F})$ verifies condition \emph{(Ca)}.
\end{lemma}

\begin{proof}
Consider a $\preceq_{F}$--decreasing sequence $\left(  (x_{n},A_{n})\right)
_{n\geq1}\subset\mathcal{A}$. Setting $\gamma_{n}:=\inf z^{\ast}(A_{n})$
$(\in\mathbb{R})$, then clearly the sequence $\left(  \gamma_{n}\right)
_{n\geq1}\subset\mathbb{R}$ is decreasing and bounded, and so $\gamma
:=\lim_{n\rightarrow\infty}\gamma_{n}\in\mathbb{R}$. Using (\ref{zv-eq3b}), we
get
\[
\gamma_{n}-\gamma_{n+1}=\inf z^{\ast}(A_{n})-\inf z^{\ast}(A_{n+1})\geq\inf
z^{\ast}\left(  F(x_{n+1},x_{n})\right)  \geq\eta\left(  d(x_{n}%
,x_{n+1})\right)  \quad\forall n\geq1.
\]
It follows that $\eta\left(  d(x_{n},x_{n+1})\right)  \rightarrow0$, and so
$d(x_{n},x_{n+1})\rightarrow0$ because $\eta:\mathbb{R}_{+}\rightarrow
\overline{\mathbb{R}}_{+}$ is increasing. Therefore, $(x_{n})_{n\geq1}$ is
asymptotic, and so, by Lemma \ref{lem2}, (Ca) is verified.
\end{proof}

\begin{theorem}
\label{zv-t2-bh}Assume that the following two conditions hold:

\emph{(i)} $F$ verifies condition \emph{(F3)},

\emph{(ii)} $(\mathcal{A},\preceq_{F})$ verifies \emph{(C1)} and $z_{F}^{\ast
}(Y_{\mathcal{A}})$ is bounded from below, where $z_{F}^{\ast}$ is provided by
\emph{(F3)}.

Then for every $(x,A)\in\mathcal{A}$ there exists a minimal element
$(\overline{x},\overline{A})\in\mathcal{A}$ wrt $\preceq_{F,z_{F}^{\ast}}$
such that $(\overline{x},\overline{A})\preceq_{F,z_{F}^{\ast}}(x,A)$; moreover
$\mathcal{A}\ni(x^{\prime},A^{\prime})\preceq_{F}(\overline{x},\overline{A})$
implies $x^{\prime}=\overline{x}$.
\end{theorem}

\begin{proof}
Taking into account (\ref{zv-f2}), $\inf z_{F}^{\ast}(F(x,x^{\prime}))\geq
\eta\left(  d(x,x^{\prime})\right)  >0$ for $x,x^{\prime}\in X$ with $x\neq
x^{\prime}$. Moreover, $\inf z_{F}^{\ast}(A)\geq\inf z_{F}^{\ast}\left(
Y_{\mathcal{A}}\right)  >-\infty$ for every $A\in\Pr_{2^{Y}}(\mathcal{A})$. On
the other hand, applying Lemma \ref{lem5}, (Ca) is verified. Since (C1) holds,
using now Proposition \ref{rel-c} we obtain that (C0) is verified. The
conclusion follows using Theorem \ref{c-hzH}~(ii).
\end{proof}

\medskip

Note that condition (F2) in Theorem \ref{zv-t2-bh} was used only for obtaining
the transitivity of $\preceq_{F}$; for $G:X\times X\rightrightarrows K$
verifying (F1) and $\mathcal{A}\subset X\times2^{Y}$ we introduce the relation
$\preceq_{G}^{\prime}$ on $\mathcal{A}$ defined by $(x_{1},A_{1})\preceq
_{G}^{\prime}(x_{2},A_{2})$ if $A_{2}\subset A_{1}+G(x_{1},x_{2})$. Of course,
for $z^{\ast}\in K^{+}$, the relation $\preceq_{G,z^{\ast}}^{\prime}$ on
$\mathcal{A}$ is introduced as in (\ref{vz-eq5}) by replacing $\preceq_{F}$
with $\preceq_{G}^{\prime}$. So, we get the following variant of Theorem
\ref{zv-t2-bh} (with the same proof).

\begin{theorem}
\label{zv-t2-LC}Assume that the following two conditions hold:

\emph{(i)} $G:X\times X\rightrightarrows K$ is such that \emph{(F1)},
\emph{(F3)} hold, and $\preceq_{G}^{\prime}$ (defined above) is transitive;

\emph{(ii)} $\mathcal{A}$ verifies \emph{(C1)} with $\preceq_{G}^{\prime}$
instead of $\preceq_{G}$, and $z_{G}^{\ast}(Y_{\mathcal{A}})$ is bounded from
below, where $z_{G}^{\ast}$ is provided by \emph{(F3)}.

Then, for every $(x,A)\in\mathcal{A}$ there exists a minimal element
$(\overline{x},\overline{A})\in\mathcal{A}$ wrt $\preceq_{G,z_{G}^{\ast}%
}^{\prime}$ such that $(\overline{x},\overline{A})\preceq_{G,z_{G}^{\ast}%
}^{\prime}(x,A)$; moreover $\mathcal{A}\ni(x^{\prime},A^{\prime})\preceq
_{G}^{\prime}(\overline{x},\overline{A})$ implies $x^{\prime}=\overline{x}$.
\end{theorem}

The next result is \cite[Th.\ 4.1]{TamZal:11} and \cite[Th.\ 10.4.4]%
{KhaTamZal:15b}.

\begin{corollary}
\label{zv-t3-Ha}Assume that $(X,d)$ is complete, $F$ satisfies condition
\emph{(F3)} and $\Gamma:X\rightrightarrows Y$ is such that $z_{F}^{\ast}$
[provided by \emph{(F3)}] is bounded from below on $\Gamma(X)$ (assumed to be
nonempty). If $S(u):=\{x\in X\mid\Gamma(u)\subseteq\Gamma(x)+F(x,u)+K\}$ is
closed for every $u\in X$, then for every $x\in\operatorname*{dom}\Gamma$
there exists $\overline{x}\in S(x)$ such that $S(\overline{x})=\{\overline
{x}\}$.
\end{corollary}

\begin{proof}
Using Lemma \ref{lem5}, we have that $(\mathcal{A}_{\Gamma},\preceq_{F})$
verifies condition (Ca), and so $(X,d,\preceq)$ satisfies condition (Aa). As
seen in Remark \ref{p-dyn}, condition (A'1) is verified. The conclusion
follows using Remark \ref{zv-rem7} and Corollary \ref{zv-t-turinici}.
\end{proof}

\begin{remark}
\label{rem5} {\rm (i) Taking into account Remark \ref{p-dyn},
instead of assuming in Corollary \ref{zv-t3-Ha} that $S(u)$ is
closed for every $u\in X$ it is sufficient to have that $S(u)$ is
$\preccurlyeq$--lower closed [that is, $(X,d,\preceq)$ verifies
(A'1), or equivalently $(\mathcal{A}_{\Gamma },\preccurlyeq_{F})$
verifies (C'1)], where $\preccurlyeq$ is defined by $x\preccurlyeq
u:\iff x\in S(u)$; moreover, instead of assuming that $(X,d)$ is
complete, it is sufficient to suppose that any $\preceq$--decreasing
Cauchy sequence is convergent [that is, $(X,d,\preceq)$ verifies
(Aa1), or equivalently $(\mathcal{A}_{\Gamma},\preccurlyeq_{F})$
verifies (Ca1)]. }

{\rm (ii) Observe that Theorem \ref{zv-t2-bh} extends \cite[Cor.\ 3.4]%
{Qiu:14} in the case in which $\Lambda$ is a singleton because condition (C'1)
is verified when $(X,d)$ is complete and the sets $S(x)$ are dynamically
closed for $x\in S(x_{0})$ (assumed to be nonempty).}
\end{remark}

With a similar proof to that of Corollary \ref{zv-t3-Ha} one gets the next
result; for $Y$ a separated locally convex space, $K\subseteq Y$ a closed
convex cone, $k^{0}\in K\backslash\{0\}$, $F(x,x^{\prime}):= \{d(x,x^{\prime
})k^{0}\}$ for $x,x^{\prime}\in X$, and $\Gamma(X)$ quasi $K$--bounded, this
reduces to \cite[Cor.\ 3.1]{LinChu09}.

\begin{corollary}
\label{zv-t3-LC}Assume that $(X,d)$ is complete, $F:X\times X\rightrightarrows
K$ satisfies conditions \emph{(F1)} and \emph{(F3)}, and $\Gamma
:X\rightrightarrows Y$ is such that $z_{F}^{\ast}$ (from \emph{(F3)}) is
bounded from below on $\Gamma(X)$. Assume also that \emph{(a)} $\{x\in
X\mid\Gamma(u)\subseteq\Gamma(x)+F(x,u)\}$ is closed for every $u\in X$, and
\emph{(b)}~$\Gamma(y)\subseteq\Gamma(x)+F(x,y)$, $\Gamma(z)\subseteq
\Gamma(y)+F(y,z)$ imply $\Gamma(z)\subseteq\Gamma(x)+F(x,z)$ for all $x,y,z\in
X$. Then for every $x\in\operatorname*{dom}\Gamma$ there exists $\overline
{x}\in X$ such that $\Gamma(x)\subseteq\Gamma(\overline{x})+F(\overline{x}
,x)$, and $\Gamma(\overline{x})\subseteq\Gamma(x)+F(x,\overline{x})+K$ implies
$x=\overline{x}$.
\end{corollary}

In the sequel $H\subset K\backslash(-K)$ is a $K$--convex set and $F_{H}$ is
defined by (\ref{zv-eqfh}). As mentioned in Lemma \ref{zv-lem1}, $F_{H}$
verifies conditions (F1) and (F2). Of course, any of the preceding results can
be reformulated for $F:=F_{H}$, $\preceq_{H}:=\preceq_{F_{H}}$ and
$\preceq_{H,z^{\ast}}:=\preceq_{F_{H},z^{\ast}}$. For example, the version of
Theorem \ref{zv-t2-bh} is the one corresponding to (ii) in the following
result. The other situations are more specific to the case $F=F_{H}$.

In the sequel all the conditions and results refer to $\preceq_{H}$.

\begin{theorem}
\label{zv-t2-bhH}Let $(\mathcal{A},\preceq_{H})$ verify condition \emph{(C1)}
wrt $\preceq_{H}$. Assume that one of the following conditions are verified:

\emph{(i)} either \emph{(a)} $0\notin\operatorname*{cl}\nolimits_{seq}(H+K)$
and $Y_{\mathcal{A}}$ is quasi $K$--bounded, or \emph{(b)} $0\notin
\operatorname*{cl}(H+K)$ and $Y_{\mathcal{A}}$ is $K$--bounded;

\emph{(ii)} there exists $z^{\ast}\in K^{+}$ such that $\inf z^{\ast}(H)>0$
and $\inf z^{\ast}(Y_{\mathcal{A}})>-\infty$,

\emph{(iii)} $H$ satisfies condition (\ref{zv-f5b}) and $Y_{\mathcal{A}}$ is
$K^{+}$--bounded,

\emph{(iv)} $H$ is seq-compact, and there exist a bounded set $B\subset Y$ and
a convex cone $C$ such that $K\subset C$, $H\subset\operatorname*{int}C$, and
$B\not \subset A+C$ for any $A\in\Pr_{2^{Y}}(\mathcal{A})$.

Then, for every $(x,A)\in\mathcal{A}$ there exists $(\overline{x},\overline
{A})\in\mathcal{A}$ such that $(\overline{x},\overline{A})\preceq_{H}(x,A)$
and $\mathcal{A}\ni(x^{\prime},A^{\prime})\preceq_{H}(\overline{x}%
,\overline{A})$ implies $x^{\prime}=\overline{x}$. Moreover, in case
\emph{(ii)}, $(\overline{x},\overline{A})$ can be taken to be a minimal
element wrt $\preceq_{H,z^{\ast}}$ such that $(\overline{x},\overline
{A})\preceq_{H,z^{\ast}}(x,A)$.
\end{theorem}

\begin{proof}
It is sufficient to show that condition (C0) is verified. Consider a
$\preceq_{H}$--decreasing sequence $\left(  (x_{n},A_{n})\right)  _{n\geq
1}\subset\mathcal{A}$ with $A_{1}\neq\emptyset$. Using Proposition \ref{lem4}
(i), (ii), (iii) or (iv) when (i), (ii), (iii) or (iv) (from our hypothesis)
holds, respectively, we obtain that $(x_{n})$ is Cauchy. By (C1), there exists
$(x,A)\in\mathcal{A}$ such that $(x,A)\preceq_{H}(x_{n},A_{n})$ for $n\geq1$.
Hence (C0) holds; the conclusion follows using Theorem \ref{c-hzH}.
\end{proof}

\medskip

With a similar proof to that of Theorems \ref{zv-t2-bhH} we obtain the next result.

\begin{theorem}
\label{zv-t3-qiu}Let $(\mathcal{A},\preceq_{H})$ verify \emph{(C1)} wrt
$\preceq_{H}$. Assume that $(x_{0},A_{0})\in\mathcal{A}$ and $\varepsilon>0$
are such that $A_{0}\not \subset A+\varepsilon H+K$ for all $A\in\Pr_{2^{Y}%
}(\mathcal{A})$. Then there exists $(\overline{x},\overline{A})\in\mathcal{A}$
such that $(\overline{x},\overline{A})\preceq_{H}(x_{0},A_{0})$, and
$\mathcal{A}\ni(x^{\prime},A^{\prime})\preceq_{H}(\overline{x},\overline{A})$
implies $x^{\prime}=\overline{x}$; moreover, $d(\overline{x},x_{0}%
)<\varepsilon$.
\end{theorem}

\begin{proof}
We apply Theorem \ref{c-hzH} to
\[
\mathcal{A}_{0}:=\left\{  (x,A)\in\mathcal{A}\mid(x,A)\preceq_{H}(x_{0}
,A_{0})\right\}
\]
preordered by $\preceq_{H}$. For this consider $\left(  (x_{n} ,A_{n})\right)
_{n\geq1}\subset\mathcal{A}_{0}$ a $\preceq_{H}$--decreasing sequence. Then,
setting $(x_{n}^{\prime}, A_{n}^{\prime}):=(x_{n-1},A_{n-1})$ $(\in
\mathcal{A}_{0})$ for $n\geq1$, the sequence $\left(  (x_{n}^{\prime}%
,A_{n}^{\prime})\right)  _{n\geq1}\subset\mathcal{A}_{0}$ is still
$\preceq_{H}$--decreasing. Since $A_{0}=A_{1}^{\prime}\not \subset
A_{n}^{\prime}+\varepsilon H+K=A_{n-1}+\varepsilon H+K$ for $n\geq1$, by
Proposition \ref{lem4}(v) we obtain that $(x_{n})_{n\geq1}$ is Cauchy. By
(C1), $\left(  (x_{n},A_{n})\right)  _{n\geq1}$ is minorized in $\mathcal{A}$,
and so in $\mathcal{A}_{0}$, too. Hence (C0) holds. Applying Theorem
\ref{c-hzH} to $(x_{0},A_{0})$ $(\in\mathcal{A}_{0})$, we get $(\overline
{x},\overline{A})\in\mathcal{A}_{0}$ such that $(\overline{x},\overline
{A})\preceq_{H}(x_{0},A_{0})$ and $\mathcal{A}_{0}\ni(x^{\prime},A^{\prime
})\preceq_{H}(\overline{x},\overline{A})$ implies $x^{\prime}=\overline{x}$.

Take $\mathcal{A}\ni(x^{\prime},A^{\prime})\preceq_{H}(\overline{x}%
,\overline{A})$; hence $(x^{\prime},A^{\prime})\in\mathcal{A}_{0}$, and so
$x^{\prime}=\overline{x}$. We have to show that $d(\overline{x},x_{0}%
)<\varepsilon$. In the contrary case, $d(\overline{x},x_{0})\geq\varepsilon$.
Taking into account that $(\overline{x},\overline{A})\preceq_{H}(x_{0},A_{0}
)$, we get
\[
A_{0}\subset\overline{A}+d(\overline{x},x_{0})H+K\subset\overline
{A}+\varepsilon H+[d(\overline{x},x_{0})-\varepsilon]H+K\subset\overline
{A}+\varepsilon H+K,
\]
contradicting the hypothesis $A_{0}\not \subset \overline{A}+\varepsilon H+K$.
The proof is complete.
\end{proof}

\medskip

Of course, the conclusions of Theorems \ref{zv-t2-bhH} and \ref{zv-t3-qiu}
remain valid if $(X,d)$ is complete and $(\mathcal{A},\preceq_{H})$ satisfies
condition (C'1) instead of (C1). Observe that condition (C'1) is introduced in
\cite[Cor.\ 35]{Ham:05} for $d$ replaced by a regular premetric $\varphi
:X\times X\rightarrow\mathbb{R}_{+}$ in the definition of $\preceq_{H}$; for
$\varphi=d$, \cite[Cor.\ 35]{Ham:05} follows from Theorem \ref{zv-t2-bhH} (i)
because (A2) implies that $0\notin\operatorname*{cl}\nolimits_{seq}(K+K^{0})$
[see Lemma \ref{lem0} (i)]. Note that the conclusion of \cite[Th.\ 4.1]%
{Qiu:16b} is valid only for that $\gamma>0$ appearing in condition
(Q3). In this case \cite[Th.\ 4.1]{Qiu:16b} follows from Theorems
\ref{zv-t2-bhH} for $H\subset D\backslash(-D)$ instead of
$0\notin\operatorname*{vcl}(H+K)$ taking
$\mathcal{A}:=\big\{(x,\{f(x)\})\mid x\in X\big\}$; moreover,
conclusion (b) implies the more usual estimate
$d(x_{0},\hat{x})<\varepsilon/\gamma$. Here, as used in
\cite{Qiu:16b}, for $A\subset Y$,
\[
\operatorname*{vcl}A=\{y\in Y\mid\exists v\in Y,~\exists(\lambda_{n}%
)\subset\mathbb{R}_{+}\text{ with }\lambda_{n}\rightarrow0,~\forall
n\in\mathbb{N}:y+\lambda_{n}v\in A\}.
\]

Theorem 5.3 of the recent paper \cite{Qiu:17} also follows from Theorem
\ref{zv-t3-qiu} applied with $X$ replaced by $S(x_{0}))$ without using the
condition $0\notin\operatorname*{vcl}(H+K)$, but just $0\notin(H+K)$ with a
(somewhat) stronger conclusion. Note also that in Theorem \ref{zv-t3-qiu} we
do not use any topology on ${Y}$ (in such a case, as mentioned in the
preliminaries, we could furnish $Y$ with the core convex topology).

Indeed, because of (ii) in \cite[Thm. 5.3]{Qiu:17}, one needs only to verify
(C1): Consider a decreasing sequence $(x_{n},f(x_{n}))$. Now, condition (ii)
implies that $(x_{n})$ is Cauchy [use Proposition \ref{lem4} under condition
(v)], and condition (i) implies that $(x_{n})$ is convergent to some $x\in X$
with $(x,f(x))\leq(x_{n},f(x_{n}))$.

The conditions below depend on the (uniformity defined by the) metric $d$, and
they do not depend on $H$. Condition (C'2) corresponds to \cite[(H2)]%
{TamZal:11} and \cite[(H2)]{KhaTamZal:15b}; they will be used for getting a
version of Theorem \ref{zv-t3-qiu} similar to the classic EVP.

\begin{description}
\item[\textbf{(C'2)}] \label{zv-H2H} $\forall\big((x_{n},A_{n})\big)_{n\geq
1}\subseteq\mathcal{A}$ with $(A_{n})_{n\geq1}$ $\leq_{K}^{l}$--decreasing and
$x_{n}\rightarrow x\in X:\exists A\in2^{Y}$ such that $(x,A)\in\mathcal{A}$
and $A\leq_{K}^{l}A_{n}$ $\forall n\geq1$,

\item[\textbf{(Ca2)}] $\forall\big((x_{n},A_{n})\big)_{n\geq1}\subseteq
\mathcal{A}$ with $(A_{n})_{n\geq1}$ $\leq_{K}^{l}$--decreasing and
$(x_{n})_{n\geq1}$ Cauchy~: $(x_{n})_{n\geq1}$ is convergent.
\end{description}

\begin{remark}
\label{rem4.6}{\rm Observe that $(X,d)$ complete implies (Ca2), and
(Ca2) implies (Ca1) wrt $\preceq_{H}$. Having a set-valued mapping
$\Gamma: X\rightrightarrows Y$, observe that
$\mathcal{A}:=\mathcal{A}_{\Gamma}$ verifies condition (C'2) iff for
every sequence $(x_{n})_{n\geq1}\subset X$ one has
$\Gamma(x)\leq^{l}_{K}\Gamma(x_{n})$ for $n\geq1$ whenever
$x_{n}\rightarrow x\in X$ and $\left(  \Gamma(x_{n})\right)
_{n\geq1}$ is $\leq^{l}_{K}$--decreasing; in such a case $\Gamma$ is
called $K$-sequentially lower monotone ($K$-s.l.m.\ for short) by
Qiu \cite[Def.\ 2.1]{Qiu:12}. Moreover, in this case (Ca2) reduces
to the $(\Gamma,K)$--lower completeness of $(X,d)$ as defined in
\cite[Def.\ 2.2]{Qiu:12} when $F=F_{H}$.}
\end{remark}

In the next result we provide several conditions; each of them together with
(C'2) implies (C'1). Recall that the set $A\subset Y$ is closed in the
direction $v\in Y$, or $v$--closed, if $y\in A$ whenever $y+\alpha_{n}v\in A$
for $n\geq1$ and $\mathbb{R}\ni\alpha_{n}\rightarrow0$; $A$ is lineally closed
if $E$ is closed in any direction $v\in Y$, or equivalently
$E=\operatorname*{vcl}E$.

Recall that our blanket assumption (H) is working, and so $A\neq\emptyset$ for
every $A\in\Pr_{2^{Y}}(\mathcal{A})$.

\begin{proposition}
\label{lem3}Let $(\mathcal{A},\preceq_{H})$ verify condition
\emph{(C'2)}. Then condition \emph{(C'1)} is verified provided one
of the following conditions holds:

\emph{(i)} $A+\lambda H+K$ is $v$-closed for some $v \in H+K$ and all
$\lambda>0$ and all $A\in\Pr_{2^{Y}}(\mathcal{A});$

\emph{(ii)} $H$ is bounded and cs-complete, and $A+K$ is seq-closed for every
$A\in\Pr_{2^{Y}}(\mathcal{A})$,

\emph{(iii)} \emph{(a)} the topology of $Y$ is generated by the
family $\mathcal{Q}$ of seminorms, \emph{(b)} either $H$ is
cs--complete or else $H$ is cs--closed and $(Y,\mathcal{Q})$ is
$\ell^{\infty }$--complete\footnote{$(Y,\mathcal{Q})$ is
$\ell^{\infty}$--complete if for every sequence
$(y_{n})_{n\geq1}\subset Y$, the series $\sum_{n\geq1}y_{n}$ is
convergent provided $\sum_{n\geq1}q(y_{n})$ is convergent for every
$q\in\mathcal{Q}$ (see \cite{Qiu:01}).}, \emph{(c)} for every $q\in
\mathcal{Q}$ there exists $z_{q}^{\ast}\in K^{+}$ such that
$q(h)\leq z_{q}^{\ast}(h)$ for every $h\in H$, and \emph{(d)} for
each $A\in\Pr_{2^{Y} }(\mathcal{A})$, $A$ is $K^{+}$--bounded and
$A+K$ is seq-closed.

\emph{(iv)} $Y$ is a reflexive Banach space, $H$ is convex and closed, $A$ is
$K$--closed and quasi $K$--bounded for every $A\in\Pr_{2^{Y}}(\mathcal{A})$,
and there exists $\alpha>0$ such that
\begin{equation}
\left\Vert h+k\right\Vert \geq\alpha\left\Vert h\right\Vert \quad\forall h\in
H,\ \forall k\in K. \label{zv-eq-ng3b}%
\end{equation}

\end{proposition}

\begin{proof}
Consider $\left(  (x_{n},A_{n})\right)  _{n\geq1}\subset\mathcal{A}$ a
$\preceq_{H}$--decreasing sequence with $x_{n}\rightarrow x\in X$. Since
$(A_{n})_{n\geq1}$ is clearly $\leq_{K}^{l}$--decreasing, by (C'2), we get
$A\in2^{Y}$ such that $(x,A)\in\mathcal{A}$ and $A\leq_{K}^{l}A_{n}$, that is
$A_{n}\subset A+K$, for $n\geq1$. We have to show that $(x,A)\preceq_{H}%
(x_{n},A_{n})$ for $n\geq1$. Clearly, using the transitivity of $\preceq_{H}$,
it is sufficient to have that $(x,A)\preceq_{H}(x_{n},A_{n})$ for $n\in P$
with $P$ an infinite subset of $\mathbb{N}\backslash\{0\}$. Because
$(x,A)\preceq_{H}(x_{n},A_{n})$ if $x_{n}=x$, we may assume that $x_{n}\neq x$
for $n\geq1$. Moreover, because $x_{n}\rightarrow x$ (passing to a subsequence
if necessary) we may assume that $\sum_{n\geq1}\delta_{n}<\infty$, where
$\delta_{n}:=d(x_{n},x_{n+1})$ $(>0)$ for $n\geq1$. Replacing $(x_{n}%
)_{n\geq1}$ with $(x_{n+p-1})_{n\geq1}$ for $p\geq1$, it sufficient to show
that $A_{1}\subset A+d(x_{1},x)H+K$. Set $\mu:=\sum_{l\geq1}\delta_{l}$,
$\mu_{0}:=0$ and $0<\mu_{p}:=\sum\nolimits_{l=1}^{p}\delta_{l}$ for $p\geq1$;
clearly, $\mu_{p}\rightarrow\mu$ for $p\rightarrow\infty$.

Take $y\in A_{1}$. Because $A_{1}\subset A_{2}+\delta_{1}H+K$, there exist
$y_{2}\in A_{2}$, $h_{1}\in H$, $k_{1}\in K$ such that $y_{1}:=y=y_{2}%
+\delta_{1}h_{1}+k_{1}$. Because $A_{2}\subset A_{3}+\delta_{2}H+K$, there
exist $y_{3}\in A_{3}$, $h_{2}\in H$, $k_{2}\in K$ such that $y_{2}%
=y_{3}+\delta_{2}h_{2}+k_{2}$. Continuing in this way we get the sequences
$(y_{p})_{p\geq1}\subset Y$, $(h_{p})_{p\geq1}\subset H$, $(k_{p})_{p\geq
1}\subset K$ such that $A_{p}\ni y_{p}=y_{p+1}+\delta_{p}h_{p}+k_{p}$ for
$p\geq1$; in particular, $(y_{p})_{p\geq1}$ is $\leq_{K}$--decreasing. Hence
\begin{equation}
y=y_{p+1}+\sum\nolimits_{l=1}^{p}\delta_{l}h_{l}+k_{p}^{\prime}=u_{p}%
+\sum\nolimits_{l=1}^{p}\delta_{l}h_{l}=u_{p}^{\prime}+\mu_{p} h_{p}^{\prime},
\label{r-Qiu2}%
\end{equation}
where $k_{p}^{\prime}:=\sum\nolimits_{l=1}^{p}k_{l}\in K$ and $u_{p}%
:=y_{p+1}+k_{p}^{\prime}\in A+K$ because $y_{p+1}\in A_{p+1}\subset A+K$, and
$u_{p}^{\prime}\in A+K$, $h_{p}^{\prime}\in H$; we used the convexity of $H+K$.

Taking into account (\ref{r-hk}), if $\mu_{p}\ge d(x_{1},x)$ for some $p\ge1$,
from the last expression of $y$ in (\ref{r-Qiu2}) we obtain that $y\in
A+d(x_{1},x)H+K$. So, we may (and do) assume that $\mu_{p}< d(x_{1},x)$ for
$p\ge1$. Consequently $\mu=d(x_{1},x)$ and $\mu_{p}=d(x_{1},x_{p+1})$ for
$p\ge1$.

Assume first that (i) holds. From (\ref{r-Qiu2}) we have that $y=u_{p}%
^{\prime}+\mu_{p} h_{p}^{\prime}$, and so
\[
y+\alpha_{p}\bar y=u_{p}^{\prime}+\mu_{p} h_{p}^{\prime}+\alpha_{p}v \in
A+K+(\mu_{p}+\alpha_{p}) H+K=A+\mu H+K,
\]
for every $p\ge1$, where $\alpha_{p}:=\mu-\mu_{p} >0$. Since $A+\mu H+K$ is
$v$-closed and $\alpha_{p}\to0$, we obtain that $y\in A+\mu H+K=A+d(x_{1},x)
H+K$.

Assume that (ii) holds. Because $H$ is bounded and cs-complete, the series
$\sum\nolimits_{l=1}^{\infty}\delta_{l}h_{l}$ converges to some $\bar v\in Y$
and $h:=\mu^{-1}\bar v\in H$. Using the second equality in (\ref{r-Qiu2}), we
get
\[
u_{p}=y-\sum\nolimits_{l=1}^{p}\delta_{l}h_{l}\rightarrow u:=y-\mu
h\in\operatorname*{cl}\nolimits_{seq}(A+K)=A+K,
\]
and so $y\in A+\mu H+K=A+d(x_{1},x) H+K$.

Assume now that (iii) holds. Fix $q\in\mathcal{Q}$ and take $z_{q}^{\ast}\in
K^{+}$ such that $q(h)\leq z_{q}^{\ast}(h)$ for $h\in H$. Because $A$ is
$K^{+}$--bounded, $\gamma_{q}:=\inf z_{q}^{\ast}(A)=\inf z_{q}^{\ast}
(A+K)\in\mathbb{R}$. From (\ref{r-Qiu2}) we get
\[
z_{q}^{\ast}(y)=z_{q}^{\ast}(u_{p})+\sum\nolimits_{l=1}^{p}\delta_{l}%
z_{q}^{\ast}(h_{l})\geq\gamma_{q}+\sum\nolimits_{l=1}^{p}\delta_{l}%
q(h_{l})=\gamma_{q}+\sum\nolimits_{l=1}^{p}q(\delta_{l}h_{l}),
\]
and so $\sum\nolimits_{l=1}^{p}q(\delta_{l}h_{l})\leq z_{q}^{\ast}%
(y)-\gamma_{q}$ for all $p\geq1$. Hence
\begin{equation}
\sum\nolimits_{l=1}^{\infty}q(\delta_{l}h_{l})<\infty\quad\forall
q\in\mathcal{Q}. \label{r-li}%
\end{equation}
This implies hat the sequence $\left(  \sum\nolimits_{l=1}^{p}\delta_{l}%
h_{l}\right)  _{p\geq1}$ is Cauchy; assuming that $H$ is cs-complete, there
exists $h\in H$ such that $\sum\nolimits_{l=1}^{\infty}\delta_{l}h_{l}=\mu h$.
Assuming that $(Y,\mathcal{Q})$ is $\ell^{\infty}$, (\ref{r-li}) implies that
the series $\sum\nolimits_{l=1}^{\infty}\delta_{l}h_{l}$ converges to some
$v\in Y$; assuming more that $H$ is cs--closed, we have again that
$h:=\mu^{-1}v\in H$. As in case (ii) above we obtain that $y\in A+d(x_{1}
,x)H+K$.

Finally, assume that (iv) holds. Taking into account (\ref{r-Qiu2}) and the
fact that $A$ is quasi $K$--bounded, we get also a bounded sequence $(b_{p}%
)_{p}\subset Y$ and the sequences $(k_{p}^{\prime\prime})_{p\geq1}%
,(k_{p}^{\prime\prime\prime})_{p\geq1}\subset K$ such that
\begin{equation}
y=y_{p+1}+\eta_{p}h_{p}^{\prime}+k_{p}^{\prime\prime}=u_{p}^{\prime}+\mu
_{p}h_{p}^{\prime}=b_{p}+\eta_{p}h_{p}^{\prime}+k_{p}^{\prime\prime\prime
}\quad\forall p\geq1.\label{r-Qiu3}%
\end{equation}

Using the last expression of $y$ in (\ref{r-Qiu3}), from (\ref{zv-eq-ng3b}) we
obtain that $\left\Vert y-b_{p}\right\Vert \geq\alpha\eta_{p}\left\Vert
h_{p}^{\prime}\right\Vert $ for $p\geq1$, and so $(h_{p}^{\prime})_{p\geq1}$
is bounded. Because $Y$ is reflexive and $H$ is (weakly) closed,
$(h_{p}^{\prime})_{p\geq1}$ has a subsequence converging weakly to $h\in H$,
and so $h\in H_{p}:=\overline{\operatorname*{conv}}\{h_{l}^{\prime}\mid l\geq
p\}$ for every $p\geq1$. Fix some $p_{1}\geq1$ such that $d(x_{p},x)<1$ for
every $p\geq p_{1}$. Because $h\in H_{p_{1}}$, there exists $p_{2}>p_{1}$ and
$(\lambda_{p}^{1})_{p_{1}\leq p<p_{2}}\subseteq\mathbb{R}_{+}$ such that
$\sum_{p=p_{1}}^{p_{2}-1}\lambda_{p}^{1}=1$ and $\left\Vert h^{1}-h\right\Vert
<1$, where $h^{1}:=\sum_{p=p_{1}}^{p_{2}-1}\lambda_{p}^{1}h_{p}^{\prime}\in
H$. Increasing if necessary $p_{2}$, we may (and do) assume that
$d(x_{p},x)<1/2$ for every $p\geq p_{2}$. Continuing in this way, we find an
increasing sequence $(p_{l})_{l\geq1}\subseteq\mathbb{N}^{\ast}$ such that for
each $l\geq1$ one has $d(x_{p},x)<1/l$ for $p\geq p_{l}$, and there exists
$(\lambda_{p}^{l})_{p_{l}\leq p<p_{l+1}}\subseteq\mathbb{R}_{+}$ such that
$\sum_{p=p_{l}}^{p_{l+1}-1}\lambda_{p}^{l}=1$ and $\left\Vert h^{l}%
-h\right\Vert <1/l$, where $h^{l}:=\sum_{p=p_{l}}^{p_{l+1}-1}\lambda_{p}%
^{l}h_{p}^{\prime}\in H$. Because $\eta_{p}=d(x_{1},x_{p+1})\geq
d(x_{1},x)-d(x_{p+1},x)$, we have that $\eta_{p}\geq d(x_{1},x)-1/l$ for every
$p\geq p_{l}$. Using the first expression of $y$ in (\ref{r-Qiu3}) and the
monotonicity of $(y_{p})_{p\geq1}$, we get
\[
y\geq_{K}y_{p+1}+\eta_{p}h_{p}^{\prime}\geq_{K}y_{p+1}+(d(x_{1},x)-1/l)h_{p}
^{\prime}\geq y_{p_{l+1}}+(d(x_{1},x)-1/l)h_{p}^{\prime}%
\]
for $p_{l}\leq p<p_{l+1}$. Multiplying by $\lambda_{p}^{l}\geq0$ and summing
up for $p_{l}\leq p<p_{l+1}$ we get $y\geq_{K}y_{p_{l+1}}+(d(x_{1}%
,x)-1/l)h^{l}$, and so $y-(d(x_{1},x)-1/l)h^{l}\in A+K$ for $l\geq1$. Passing
to the limit for $l\rightarrow\infty$ we obtain that $y-d(x_{1},x)h\in A+K$
because $A+K$ is closed. Hence $y\in A+d(x_{1},x)H+K$.
\end{proof}

\medskip

Note that $0\notin\operatorname*{cl}(H+K)$ when condition (iii)~(c) in
Proposition \ref{lem3} is verified.

In the next result, we provide several conditions ensuring the $v$-closedness
of the set $A+\lambda H +K$ appearing in condition (i) of Proposition
\ref{lem3}.

\begin{proposition}
\label{lem3b}Let $A\subset Y$ be nonempty, $v\in Y\backslash\{0\}$, and
$\lambda>0$. Then $A+\lambda H +K$ is $v$-closed provided one of the following
conditions holds:

\emph{(i)} $H$ is seq-compact, and $A+K$ is seq-closed;

\emph{(ii)} $H+K$ is seq-closed, and $A$ is seq-compact;

\emph{(iii)} $H$ is a singleton, and $A+K$ is $v$--closed;

\emph{(iv)} $H+K$ is $v$-closed, and $A$ is finite.
\end{proposition}

\begin{proof}
As already observed, any sequentially closed set is $v$-closed.

(i) Clearly, $\lambda H$ is seq-compact, and so, using Lemma \ref{lem0}~(i),
$A+\lambda H+K$ $[=(A+K)+\lambda H]$ is seq-closed.

(ii) Since $\lambda(H+K)=\lambda H+K$ and the first set is seq-closed, we
obtain that $A+\lambda H+K$ is seq-closed as in (i).

(iii) The assertion is obvious.

(iv) Since $H+K$ is $v$-closed, $\lambda H+K$ $[=\lambda(H+K)]$ is $v$-closed,
and so $u+\lambda H+K$ is $v$-closed for every $u\in A$. The conclusion
follows from the fact that the union of a finite family of $v$-closed sets is
$v$-closed.
\end{proof}

\begin{remark}
\label{rem3}{\rm When $(X,d)$ is complete, in Theorems
\ref{zv-t2-bhH} and \ref{zv-t3-qiu} one can replace the hypothesis
that (C1) holds with (C'2) together with one of the conditions
(i)--(iv) from Proposition \ref{lem3}. Even more, instead of
assuming that $(X,d)$ is complete in the resulting statement one can
assume that (Ca2) or even (Ca1) holds. Moreover, when one needs to
have the conclusion for a given $(x_{0},A_{0})\in\mathcal{A}$
(instead of any $(x,A)\in\mathcal{A}$), one may replace
$\mathcal{A}$ with
$\mathcal{A}_{0}:=\{(x,A)\in\mathcal{A}\mid(x,A)\preceq_{H}(x_{0},A_{0})\}$
in the hypothesis of the respective statement.}
\end{remark}

We exemplify (partially) Remark \ref{rem3} with the next result.

\begin{theorem}
\label{zv-c3-qiu}Let $(\mathcal{A},\preceq_{H})$ verify conditions
\emph{(C'2)} and \emph{(Ca1)}, as well as one of the conditions
\emph{(i)--(iv)} of Proposition \ref{lem3}. Assume that $(x_{0},A_{0}%
)\in\mathcal{A}$ and $\varepsilon>0$ are such that $A_{0}\not \subset
A+\varepsilon H+K$ for all $A\in\Pr_{2^{Y}}(\mathcal{A})$. Then for every
$\lambda>0$ there exists $(x_{\lambda},A_{\lambda})\in\mathcal{A}$ such that
\emph{(a)} $A_{0}\subset A_{\lambda}+\lambda d(x_{\lambda},x_{0})H+K$,
\emph{(b)} $d(x_{\lambda},x_{0})<\varepsilon/\lambda$, \emph{(c)} $A_{\lambda
}\not \subset A+\lambda d(x,x_{\lambda})H+K$ for every $(x,A)\in\mathcal{A}$
with $x\neq x_{\lambda}$.
\end{theorem}

\begin{proof}
Set $d^{\prime}:=\lambda d$ with $\lambda>0$. Of course, for $(x_{n})_{n\ge
1}\subset X$ and $x\in X$ we have that $(x_{n})$ is $d$--Cauchy (resp.\ $x_{n}
\overset{d}{\rightarrow}x$) if and only if $(x_{n})$ is $d^{\prime}$--Cauchy
(resp.\ $x_{n}\overset{d^{\prime}}{\rightarrow}x$). Moreover, condition (*)
(among (i)--(iv)) is verified wrt $d$ iff (*) is verified wrt $d^{\prime}$. By
Proposition \ref{lem3}, $\mathcal{A}$ verifies (C'1) wrt $d^{\prime}$, and so,
by (\ref{r-c0011}), $\mathcal{A}$ verifies (C1) wrt $d^{\prime}$. Applying
Theorem \ref{zv-t3-qiu} we get the conclusion.
\end{proof}

\medskip

It is an easy matter to adapt the preceding results for $\mathcal{A}_{\Gamma}$
with $\Gamma:X\rightrightarrows Y$ because, as seen in Remark \ref{rem4.6},
$\mathcal{A}_{\Gamma}$ verifies (C'2) exactly when $\Gamma$ is $K$--s.l.m.,
$\mathcal{A}_{\Gamma}$ verifies (Ca2) wrt $\preceq_{H}$ exactly when $(X,d)$
is $(\Gamma,K)$--lower complete in the sense of \cite[Def.\ 2.2]{Qiu:12}, and
$\mathcal{A}_{0}:=\{(x,\Gamma(x))\mid x\in X,\ \Gamma(x_{0})\subset
\Gamma(x)+d(x,x_{0})H+K\}$ verifies (Ca1) wrt $\preceq_{H}$ exactly when
$(X,d)$ is $S(x_{0})$--dynamically complete in the sense of \cite[Def.\ 3.1]%
{Qiu:14}. So, we get the next version of Theorem \ref{zv-c3-qiu}.

\begin{corollary}
\label{c-qiu}Let $\Gamma:X\rightrightarrows Y$ be $K$--s.l.m., and let
$\mathcal{A}_{\Gamma}:=\left\{  (x,\Gamma(x))\mid x\in X\right\}  $ satisfy
\emph{(Ca1)} wrt $\preceq_{H}$. Assume that $x_{0}\in\operatorname*{dom}%
\Gamma$ and $\varepsilon>0$ are such that $\Gamma(x_{0})\not \subset
\Gamma(x)+\varepsilon H+K$ for all $x\in X$, and $\mathcal{A}_{\Gamma}$
verifies one of the conditions \emph{(i)--(iv)} of Proposition \ref{lem3}.
Then for every $\lambda>0$ there exists $x_{\lambda}\in X$ such that
$\Gamma(x_{0})\subset\Gamma(x_{\lambda})+\lambda d(x_{0},x_{\lambda})H+K$,
$d(x_{0},x_{\lambda})<\varepsilon/\lambda$, and $\Gamma(x_{\lambda
})\not \subset \Gamma(x)+\lambda d(x,x_{\lambda})H+K$ for every $x\in
X\backslash\{x_{\lambda}\}$.
\end{corollary}

\begin{remark}
\label{rem-qiu} {\rm Using Corollary \ref{c-qiu} one can obtain the
following results: }

{\rm (a) \cite[Cor.\ 3.5]{Qiu:14} for $\Gamma:=f$, $K:=D$ and
$H:=\{k_{0}\}\subset K\backslash(-K)$, because condition (i) of
Proposition \ref{lem3} [via Proposition \ref{lem3b}~(iii)] holds in
this case; moreover, the assumption that $f(X)$ is quasi
$D$--bounded is superfluous. }

{\rm (b) \cite[Cor.\ 3.6]{Qiu:14} (for $\lambda$ mentioned in the
definition of $S(x_{0})$) for $\Gamma:=f$, $K:=D$ and
$H:=\{k_{0}\}\subset K\backslash(-K)$ (instead of $k_{0}\in
K\backslash(-\operatorname*{vcl}K)$), because condition (i) of
Proposition \ref{lem3} [via Proposition \ref{lem3b}~(iii)] holds in
this case; moreover, the assumption that $f(x_{0})\not \subset
f(X)+\varepsilon k_{0}+D$ can be replaced by $f(x_{0})\not \subset
f(x)+\varepsilon k_{0}+D$ for every $x\in X$. }

{\rm (c) \cite[Th.\ 4.2]{Qiu:16b} for $\Gamma(x):=\{f(x)\}$, $K:=D$
and $H\subset K\backslash(-K)$ (instead of
$0\notin\operatorname*{vcl}(H+K)$), because condition (i) of
Proposition \ref{lem3} [via Proposition \ref{lem3b}~(iv)] holds in
this case. }

{\rm (d) As shown in its proof, \cite[Th.\ 4.3]{Qiu:16b} follows
from \cite[Th.\ 4.2]{Qiu:16b} because $H\subset D\backslash(-D)$ and
$H$ being $\sigma(Y,D^{+})$--countably compact imply
$0\notin\operatorname*{vcl}(H+D)$. In fact, the same proof shows
that in the previous implication one may replace
$\operatorname*{vcl}(H+D)$ by $\operatorname*{cl}_{seq}(H+D)$.}
\end{remark}

Similar to Corollary \ref{c-qiu}, the next result is a reformulation of
Theorem \ref{zv-t2-bhH}; for getting its conclusion apply Theorem
\ref{zv-t2-bhH} for $\mathcal{A}:=\mathcal{A}_{\Gamma}$, observing that (C'1) holds.

\begin{corollary}
\label{zv-t4-Ha}Let $\Gamma:X\rightrightarrows Y$ be $K$--s.l.m., and let
$\mathcal{A}_{\Gamma}$ satisfy \emph{(Ca1)} wrt $\preceq_{H}$. Assume that
$\mathcal{A}_{\Gamma}$ verifies one of the conditions \emph{(i)--(iv)} of
Proposition \ref{lem3}. Furthermore, suppose that one of the following
conditions holds:

\emph{(i)} either \emph{(a)} $0\notin\operatorname*{cl}\nolimits_{seq}(H+K)$
and $\Gamma(X)$ is quasi $K$--bounded, or $0\notin\operatorname*{cl}(H+K)$ and
$\Gamma(X)$ is $K$--bounded;

\emph{(ii)} there exists $z^{\ast}\in K^{+}$ such that $\inf z^{\ast}
(\Gamma(X))>-\infty$ and $\inf z^{\ast}(H)>0$, \emph{(iii)} $H$ satisfies
condition (\ref{zv-f5b}) and $\Gamma(X)$ is $K^{+}$--bounded,

\emph{(iv)} $H$ is seq-compact, and there exist a bounded set $B\subset Y$ and
a convex cone $C$ such that $K\subset C$, $H\subset\operatorname*{int}C$, and
$B\not \subset \Gamma(x)+C$ for any $x\in X$, Then for every $x\in
\operatorname*{dom}\Gamma$ there exists $\overline{x}\in S(x)$ such that
$S(\overline{x})=\{\overline{x}\}$, where $S(u):=\{u^{\prime}\in X\mid
\Gamma(u)\subseteq\Gamma(u^{\prime})+d(u,u^{\prime})H+K\}$ for $u\in X$.
\end{corollary}

\begin{remark}
\label{rem4} {\rm Replacing $d$ by $\gamma d$, Corollary
\ref{zv-t4-Ha} (ii) covers \cite[Ths.~4.2 and 4.2']{Qiu:14} because
$f$ is $D$-s.l.m.\ (hence (C'2) is verified), and conditions (ii)
and (i) of Proposition \ref{lem3} are satisfied, respectively; in
\cite[Ths.~4.2']{Qiu:14} one has (B$_{2}^{\prime} $) $\Rightarrow$
(B$_{1}^{\prime}$), and the boundedness of $H$ is not needed.
Corollary \ref{zv-t4-Ha} (i) covers also \cite[Th.\ 4.3]{Qiu:14}
because $f$ is $D$-s.l.m.\ and condition (iii) of Proposition
\ref{lem3} is satisfied.}
\end{remark}

It was observed in \cite[p.~921]{Qiu:14} that $\Gamma$ is $K$-s.l.m.\ and has
$K$--closed values when $\operatorname*{epi}\Gamma:=\{(x,y)\in X\times Y\mid
y\in f(x)+K\}$ is closed. Using Corollary \ref{zv-t4-Ha}~(i), this shows that
the conclusion of \cite[Th.~10.4.9]{KhaTamZal:15b} remains true. Indeed, (i)
of Corollary \ref{zv-t4-Ha} is verified; moreover, when condition (i), (ii),
or (iii) of \cite[Th.~10.4.9]{KhaTamZal:15b} holds, then condition (ii),
(iii), or (iv) of Proposition \ref{lem3} is verified, respectively. Of course,
even this variant of \cite[Th.~10.4.9]{KhaTamZal:15b} can be strengthened
replacing the quasi boundedness of $\Gamma(X)$ by the quasi boundedness of
each $\Gamma(x)$ with $x\in X$; moreover, condition (ii) of \cite[Th.~10.4.9]%
{KhaTamZal:15b} could be replaced by conditions (iii)~(a)--(c) from
Proposition \ref{lem3}.


\bigskip
{\bf Acknowledgement.} The work of the first author was part of the research project "Set Optimization via Abstract Convexity" funded by the Free University of Bozen-Bolzano, the work of the second author was funded by CNCS-UEFISCDI (Romania) under grant number PN-III-P4-ID-PCE-2016-0188.

\end{document}